\documentclass[11pt]{article}

\title{On the wavelets-based SWIFT method for backward stochastic differential equations}
\author{Ki Wai Chau and Cornelis W. Oosterlee}

\usepackage{fullpage}
\usepackage{amsfonts,amssymb,amsmath,amsthm,mathdots,bbm}
\usepackage{graphicx}
\usepackage{caption}
\usepackage{subcaption}
\usepackage{color}

\usepackage[boxed]{algorithm2e}

\allowdisplaybreaks[1]
\numberwithin{equation}{section}

\newtheorem{theorem}{Theorem}[section]

\theoremstyle{definition}\newtheorem{definition}[theorem]{Definition}
\theoremstyle{remark}\newtheorem{remark}[theorem]{Remark}

\begin{document}
\maketitle

\begin{abstract}
We propose a numerical algorithm for backward stochastic differential equations based on time discretization and trigonometric wavelets. This method combines the effectiveness of Fourier-based methods and the simplicity of a wavelet-based formula, resulting in an algorithm that is both accurate and easy to implement. Furthermore, we mitigate the problem of errors near the computation boundaries by means of an antireflective boundary technique, giving an improved approximation. We test our algorithm with different numerical experiments.
\end{abstract}

\section{Introduction}
Ever since the general notion of backward stochastic differential equations (BSDEs) was introduced in \cite{pardoux_peng_1990}, it has been a popular research subject. Especially in mathematical finance and insurance, BSDEs are powerful tools for the valuation of a contingent claim, both under usual complete market setting or with the incorporation of market imperfections and  collateral requirements.

Finding an analytic solution for such equations is often difficult or even impossible, either by directly solving or by transforming the problems into partial differential equations (see \cite{pardoux_peng_1992}). Therefore, numerical methods are in great demand. While the majority of so-called probabilistic methods for solving BSDEs relies on time discretization of the stochastic process, they differ by the methods for calculating the appearing conditional expectations. Techniques used include least-squares Monte Carlo regression in \cite{lemor_gobet_warin_2006}, chaos decomposition formulas in \cite{briand_labert_2014}, cubature methods in \cite{crisan_manolarakis_2012}, among others. In particular, we are interested in methods based on Fourier series. These methods link the computation of the expectations to the characteristic function of the transitional probability density function, which is either given or easy to approximate. One particular method is the BCOS method proposed in \cite{ruijter_oosterlee_2015}, which was derived from the COS method for option pricing in \cite{fang_oosterlee_2009_original_COS}. 

There have been new developments in  option pricing methods based on Fourier series. The Shannon Wavelets Inverse Fourier Technique (SWIFT method) was proposed in \cite{ortiz_gracia_oosterlee_2016} for pricing European options and a so-called {\em quick SWIFT variant} was developed in \cite{maree_ortiz-gracia_oosterlee_2016} for pricing American and barrier options. The quick SWIFT method, while also based on Shannon wavelets, has the additional benefit of simplifying the algorithm and the error formula. Moreover, it is much easier to adjust individual approximation values because wavelets form a localized basis. We propose a new approach to solving BSDEs by combining a general $\theta$-method for time-integration, as used in \cite{iserles_2008} and \cite{ruijter_oosterlee_2015}, with the SWIFT method. We also improve on previous work on SWIFT by providing an alternative derivation that takes into account the computational range.

In Section \ref{backward_stochastic_differential_equations}, the class of BSDEs under our consideration along with some notations and standing assumptions will be introduced. Section \ref{swift_method} contains the derivation of the SWIFT formula and our numerical algorithm for the BSDEs, while Section \ref{section:errors_and_computational_complexity} is related to the error and computational complexity of our algorithm. We further improve our algorithm along the computation boundary in Section \ref{section:anti-reflective_boundary}. Various numerical experiments are performed in Section \ref{section:numerical_experiments} and concluding remarks are given in Section \ref{section:conclusion}.

\section{Backward stochastic differential equations} \label{backward_stochastic_differential_equations}
\subsection{Setting}
Given a filtered complete probability space $(\Omega, \mathcal{F},\mathbb{F},\mathbb{P})$ with $\mathbb{F}:=(\mathcal{F}_t)_{0\leq t\leq T}$ a filtration satisfying the usual conditions for a fixed terminal time $T > 0$. The process $\omega:=(\omega_t)_{0\leq t\leq T}$ is a Brownian motion adapted to the filtration $\mathbb{F}$ and we are interested in solving the following one-dimension decoupled forward-backward stochastic differential equations, known as FBSDEs, numerically.   
\begin{equation}
\left\{
	\begin{array}{l}
		dX_t = \mu(t, X_t) dt +\sigma(t, X_t) d\omega_t;\\
		dY_t=-f(t,X_t, Y_t, Z_t) dt+ Z_t d\omega_t, 
	\end{array}
\right.\label{BSDE}
\end{equation}
where $0\leq t\leq T$. The functions $\mu:\Omega \times [0,T] \times \mathbb{R}\rightarrow \mathbb{R}$ and $\sigma: \Omega \times [0,T] \times \mathbb{R} \rightarrow \mathbb{R}$ refer to the drift and the diffusion coefficients of the forward stochastic process, $X$, and $x_0\in\mathcal{F}_0$ is the initial condition for $X$. The function $f:\Omega\times[0,T]\times \mathbb{R} \times \mathbb{R} \times \mathbb{R}$ is called the driver function of the backward process and the terminal condition $Y_T$ is given by $g(X_T)$ for a function $g:\Omega\times\mathbb{R}\rightarrow\mathbb{R}$. All stochastic integrals with $\omega$ are of the It\^o type.

It is assumed that both $\mu(t,x)$ and $\sigma(t,x)$ are measurable functions that are uniformly Lipschitz in $x$ and satisfy a linear growth condition in $x$. 
Therefore, there exists a unique strong solution for the forward stochastic differential equation, $$X_t = x_0 + \int^t_0 \mu(\tau, X_\tau) d\tau	+ \int^t_0\sigma(\tau, X_\tau) d\omega_\tau.$$ 
This process also satisfies the Markov property, namely $\mathbb{E}[X_\tau|\mathcal{F}_t] = \mathbb{E}[X_\tau|X_t]$ for $\tau \geq t$, where $\mathbb{E}[\cdot]$ denotes expectation with respect to probability measure $\mathbb{P}$. 

A pair of adapted processes $(Y,Z)$ is said to be the solution of the FBSDE if $Y$ is a continuous real-value adapted process, $Z$ is a real-value predictable process such that $\int^T_0|Z_t|^2dt < \infty$ almost surely in $\mathbb{P}$ and the pair satisfies Equation \eqref{BSDE}.
We wish to find $(Y_0,Z_0)$ by solving the problem backwards in time. We do this by first discretizing Equation \eqref{BSDE} along the time direction $\Lambda: 0=t_0<t_1<t_2<\ldots<t_P = T$. In this article, we assume that we have a fixed uniform time-step $\Delta t = t_{p+1}-t_{p},\, \forall p$ and define $\Delta \omega_{p+1} := \omega_{t_{p+1}}-\omega_{t_p}\sim\mathcal{N}(0,\Delta t)$, a normally distributed process. 
The discretized forward process $X^\Delta$ is defined by $$X^\Delta _{t_0} := x_0,\, X^\Delta_{t_{p+1}} := X^\Delta_{t_p} + \mu(t_p, X^\Delta_{t_p}) \Delta t +\sigma(t_p,X^\Delta_{t_p}) \Delta \omega_{p+1},\quad p=0,\ldots ,P-1,$$ which is derived from the classical Euler discretization.
Note that we only defined the discretized process at the discrete time points here. 
While it is possible to extend the definition to $[0,T]$, it is not necessary for our presentation.

Adopting the notation ${\bf X} = (X,Y,Z)$, we can observe from the backward equation,
\begin{equation}
Y_{t_p} = Y_{t_{p+1}}+\int^{t_{p+1}}_{t_p}f(\tau, {\bf X}_\tau)d\tau-\int^{t_{p+1}}_{t_p}Z_\tau d\omega_\tau,
\label{backward_time_period}
\end{equation} 
that a simple discretization is not sufficient to produce an approximation. 
It is because we would require the value of $Y_{t_{p+1}}$ to approximate $Y_{t_p}$, but $Y_{t_{p+1}}$ is not $\mathcal{F}_{t_p}$ adapted. 
To tackle this problem, we follow the standard methods in literature, for example, in \cite{bouchard_touzi_2004}. By taking conditional expectations on both sides of Equation \eqref{backward_time_period} and approximating the time integral
by a $\theta$-time discretization, as in \cite{ruijter_oosterlee_2015}, we get
\begin{align}
Y_{t_p} 
= & \;
\mathbb{E}_p[Y_{t_{p+1}}]+\int^{t_{p+1}}_{t_p} \mathbb{E}_p[f(\tau,{\bf X}_\tau)] d\tau\nonumber\\
 \approx & \;
\mathbb{E}_p[Y_{t_{p+1}}] 
+ \Delta t \theta_1 f(t_p, {\bf X}_{t_p})
+ \Delta t (1-\theta_1) \mathbb{E}_p[f(t_{p+1},{\bf X}_{t_{p+1}})],\,
\theta_1\in[0,1].\nonumber
\end{align}
The notation $\mathbb{E}_p$ and $\mathbb{E}^x_p$ are defined as $$\mathbb{E}_p[\cdot]:=\mathbb{E}[\cdot|X_{t_p}]=\mathbb{E}[\cdot|\mathcal{F}_{t_p}], \mbox{ and }\mathbb{E}^x_p[\cdot] = \mathbb{E}[\cdot|X_{t_p}=x].$$ For the process $Z$, we derive a recursive approximation formula by multiplying $\Delta \omega_{p+1}$ to both sides of Equation \eqref{backward_time_period} and taking conditional expectations,   
\begin{align}
0 
= &
\mathbb{E}_p[Y_{t_{p+1}} \Delta \omega_{p+1}] + \int^{t_{p+1}}_{t_p}\mathbb{E}_p[f(\tau,{\bf X}_\tau)\Delta \omega_{p+1}]d\tau
-\int^{t_{p+1}}_{t_p}\mathbb{E}_p[Z_\tau]d\tau \nonumber\\
\approx & 
\mathbb{E}_p[Y_{t_{p+1}} \Delta \omega_{p+1}]
+ \Delta t (1-\theta_2) \mathbb{E}_p[f(t_{p+1},{\bf X}_{t_{p+1}}) \Delta \omega_{p+1}]\nonumber\\
&- \Delta t \theta_2 Z_{t_p} - \Delta t (1-\theta_2)\mathbb{E}_p[Z_{t_{p+1}}],\; 
\theta_2\in(0,1].\nonumber
\end{align}
Again, we applied the $\theta$-method to the time integral. 
However, the two parameters for the $\theta$-method, $\theta_1$ and $\theta_2$, need not necessarily be the same. We define a discrete-time approximation $(Y^\Delta,Z^\Delta)$ for $(Y,Z)$:
\begin{subequations}
\begin{align}
&\;
Y^\Delta_{t_P} := g(X^\Delta_{t_P}), \quad 
Z^\Delta_{t_P}=\sigma(t_P,X^\Delta_{t_P})D_xg(X^\Delta_{t_P}),\label{terminal_conditions}\\
&
\mbox{for } p = P-1,\ldots,0,\nonumber\\
&\;
\;\; Z^\Delta_{t_p} := -\frac{1-\theta_2}{\theta_2}\mathbb{E}_p[Z^\Delta_{t_{p+1}}]
+\frac{1}{\theta_2 \Delta t}\mathbb{E}_p[Y^\Delta_{t_{p+1}}\Delta\omega_{p+1}]
+\frac{1-\theta_2}{\theta_2}\mathbb{E}_p[f(t_{p+1},{\bf X}^\Delta_{t_{p+1}}) \Delta \omega_{p+1}],\label{discrtized_z}\\
&\;
\;\; Y^\Delta_{t_p} := \mathbb{E}_p[Y^\Delta_{t_{p+1}}]
+ \Delta t \theta_1 f(t_p,{\bf X}^\Delta_{t_p}) 
+ \Delta t (1-\theta_1)\mathbb{E}_p[f(t_{p+1},{\bf X}^\Delta_{t_{p+1}})],\label{discrtized_y}
\end{align}
\label{discretized}
\end{subequations}
Again, we used the simplifying notation ${\bf X}^{\Delta} = (X^{\Delta}, Y^\Delta, Z^\Delta)$.
Note that various combinations of $\theta_1$ and $\theta_2$ give different approximation schemes. 
We have an explicit scheme for $Y^\Delta$ if $\theta_1=0$, and an implicit scheme otherwise. 
The variable $Z^\Delta_{t_p}$ depends on $\mathbb{E}_p[Z^\Delta_{t_{p+1}}]$ only if $\theta_2\neq1$. 
Also, since the terminal processes $Y^\Delta_{t_P}$ and $Z^\Delta_{t_P}$ are deterministic with respect to $X^\Delta_{t_P}$ and $X^\Delta$ is a Markov process, one can show by induction that $Y^\Delta_{t_p} = y^\Delta_p(X^\Delta_{t_p})$, $Z^\Delta_{t_p} = z^\Delta_p(X^\Delta_{t_p})$, where $z^\Delta_p$ and $y^\Delta_p$ are deterministic functions related to the discretization scheme. 
We shall use the notation $(y^\Delta_p(x),z^\Delta_p(x))$ when we wish to emphasize the dependence of our approximation.

When solving the approximation in Equation \eqref{discretized}, one needs to calculate multiple conditional expectations at each time-step. In this article, our method of choice is a wavelet-based method, introduced in \cite{ortiz_gracia_oosterlee_2016}.

\subsection{Standing assumptions}
Throughout the paper, in addition to the conditions for $\mu$ and $\sigma$, we assume the following to be true:
\begin{itemize}
\item[(A1)] The function $f(t,x,y,z)$ is continuous with respect to $(x,y,z)$ and all one-sided derivatives exist.
\item[(A2)] The function $g(x)$ is continuous in $x$ and all left- and right-side derivatives exist. 
\end{itemize}
When dealing with the discretization scheme with $\theta_1\neq 1$, we add one more assumption: 
\begin{itemize}
\item[(A3)] The function $f$ is Lipschitz in $(y,z)$, namely, $$|f(t,x,y_1,z_1)-f(t,x,y_2,z_2)| \leq M (|y_1-y_2|+|z_1-z_2|); \;x,y_1,y_2,z_1,z_2\in\mathbb{R}, t\in[0,T],$$ for some constant $M$.
\end{itemize}
Under assumptions (A1)-(A3) ((A1)-(A2) if $\theta_1 = 1$), the numerical algorithm for the FBSDE, which will be given in Section \ref{swift_method}, is well-defined. 
Although, $D_xg$ in Equation \eqref{terminal_conditions} may be undefined at countable many distinctive points, it can just be replaced by a one-sided derivative at these points. 
The conditions above can also ensure satisfactory performance of our algorithms in general, with more details coming in Section \ref{section:errors_and_computational_complexity}. 
However, the above conditions are not sufficient to assure the existence of the pair of adapted processes $(Y,Z)$, which is the foundation of any  numerical algorithm. We introduce an extra assumption to ensure the existence and uniqueness of the solution $(Y,Z)$ to Equation \eqref{BSDE}.
\begin{itemize}
\item[(A4)] There exists a constant $M$ such that $$|f(t,x,y,z)|+|g(x)|\leq M(1+|x|^\nu+|y|+|z|),\; \forall x,y,z\in \mathbb{R}, t\in[0,T], \nu\geq\frac{1}{2}.$$  
\end{itemize}
For further results on the existence and uniqueness of the solution of BSDEs, readers are referred to \cite{pardoux_peng_1992} and further research extending this result. The last point we would like to raise is that the convergent rate of the discretized process to the original process also depends on the functions $\mu$, $\sigma$, $f$ and $g$. We shall discuss these requirements in Section \ref{discretization_error_of_the_BSDE}; these conditions are not included in the standing assumptions.

\section{SWIFT method}\label{swift_method}
For the computation of the expectations appearing in the discrete FBSDEs (\ref{discretized}), we will use the wavelet-based SWIFT method. In this section, we first provide an alternative derivation for the SWIFT formula used in \cite{ortiz_gracia_oosterlee_2016} and \cite{maree_ortiz-gracia_oosterlee_2016}. Instead of using an approximation space based on Shannon wavelets on the whole real line, we construct a Shannon wavelet scaling function on a {\em finite domain} and derive our formula with this scaling function. This  approach is beneficial since a truncation of the integration range is required when calculating the wavelet coefficients for the SWIFT formula. Incorporating the truncated range in the scaling function simplifies the formula for the  approximation error. Next, we apply the SWIFT method to compute the conditional expectations in the discrete-time approximation of the FBSDE in Equation \eqref{discretized} and produce an algorithm for solving FBSDEs recursively, backwards in time. In Sections \ref{scaling_functions} and \ref{quick_SWIFT_formula_and_coefficients}, we derive the SWIFT formula with the finite range approach and compute the relevant expectations for the FBSDE algorithm. Section \ref{quick_SWIFT_approximation_of_function_z} and \ref{quick_SWIFT_approximations_of_function_y} discuss the approximations of the functions $z^\Delta_p(x)$ and $y^\Delta_p(x)$.

\subsection{Scaling Functions}\label{scaling_functions}
We begin our discussion with some preliminary definitions and results. For any fixed real number $m$ and integer $J \neq 0$, we define an inner product and a norm:
\begin{equation*}
<v,w>:=\frac{2^m}{J}\int^{2^{-m}J}_{-2^{-m}J}v(x)w(x) dx, \quad
||v||_2:=\sqrt{<v,v>}.
\end{equation*}  
A function $v$ is said to be in the $L^2((-2^{-m}J,2^{-m}J])$ space if $||v||_2$ is a finite number. It can be shown that  the set $$\Gamma_{m,J}:=\left\{\left.\cos\left(\left(\frac{2n-1}{2J}\pi\right)2^m x\right),\sin\left(\left(\frac{2n-1}{2J}\pi\right)2^m x\right)\right| n=1,2,\ldots\right\},$$ is orthonormal with respect to this inner product and is dense in $L^2((-2^{-m}J,2^{-m}J])$.

Equipped with the above definitions, we construct an approximation space together with a localized basis, which are the foundations of the truncated SWIFT approximation method. Consider $2J$ distinctive functions $\varphi_{J,r}:\mathbb{R}\rightarrow\mathbb{R}$,
\begin{align}
\varphi_{J,r} (x)
:=& \sum^J_{k=1}\left(
	\cos\left(\left(\frac{2k-1}{2J}\pi\right) 2^m x\right)
	\cos\left(\left(\frac{2k-1}{2J}\pi\right) 2^m \left(\frac{r}{2^m}\right)\right)
\right.\nonumber\\
&\qquad\left.
	+\sin\left(\left(\frac{2k-1}{2J}\pi\right) 2^m x\right)
	\sin\left(\left(\frac{2k-1}{2J}\pi\right) 2^m \left(\frac{r}{2^m}\right)\right)
\right)\nonumber\\
=& \sum^J_{k=1}\cos\left(
	\frac{2k-1}{2J}\pi
	\left(2^m x-r\right)
\right)\label{kernel}\\
=& \left\{
	\begin{array}{ll}
		J&\mbox{ if } x = \frac{2J}{2^m}l +\frac{r}{2^m} \mbox{ for } l \mbox{ an even integer,}\\
		-J&\mbox{ if } x = \frac{2J}{2^m}l +\frac{r}{2^m} \mbox{ for } l \mbox{ an odd integer,}\\
		\frac{\sin\left(\pi (2^m x-r)\right)}{2\sin\left(\frac{\pi}{2 J}(2^m x -r)\right)}	& \mbox{ otherwise,}
	\end{array}
\right.\nonumber
\end{align}
where $r=1-J, 2-J, \ldots, J$. This definition is a special case of the scaling functions given in Equation (2.13) of  \cite{fischer_prestin_1997}, in which the authors presented a uniform approach for the construction of wavelets based on orthogonal polynomials. The properties of $\varphi_{J,r}$ have been extensively studied in \cite{fischer_prestin_1997} and those that are relevant to our numerical method are listed in the next theorem along with their proof.
\begin{theorem}
The scaling function $\varphi_{J,r}$, which is defined in Equation \eqref{kernel}, satisfies the following properties:
\begin{enumerate}
\item[(a)] The inner product between two scaling functions is given by the following equation:
\begin{equation*}
<\varphi_{J,r},\varphi_{J,s}> = \varphi_{J,r}\left(\frac{s}{2^m}\right), \qquad r,s = 1-J,2-J, \ldots, J.
\end{equation*}
Thus, $\{\varphi_{J,r}|r = 1-J,2-J, \ldots, J\}$ form an orthogonal set.

\item[(b)] The scaling function $\varphi_{J,r}$ is localized around $\frac{r}{2^m}$. By this we mean that for the subspace
$$V_J:=\mbox{span}\left\{
	\left.
		\cos\left(\frac{(2k-1)\pi}{2J}2^m x\right),
		\sin\left(\frac{(2k-1)\pi}{2J}2^m x\right)
	\right|
	k = 1, 2, \ldots, J
\right\},$$
we have 
$$\left|\left|
	\frac{\varphi_{J,r}}{\varphi_{J,r}(2^{-m}r)}
\right|\right|_2
=\min\left\{||\chi||_2:\chi\in V_J, \chi\left(\frac{r}{2^m}\right)=1\right\} .$$

\item[(c)] $\{\varphi_{J,r}|r=1-J,2-J,\ldots,J\}$  is a basis for $V_J$.

\item[(d)] The scaling function $\varphi_{J,r}$ is also a kernel polynomial in the sense that for any function $v$ in $V_j$, we have $$<v,\varphi_{J,r}> = v\left(\frac{r}{2^m}\right).$$
\end{enumerate}
\begin{proof}
We can demonstrate (a) by direct computation and applying the orthonormality of the set $\Gamma_{m,J}$, such that 
\begin{align*}
<\varphi_{J,r},\varphi_{J,s}>
=& \sum^n_{k=1}\left(
	\cos\left(\left(\frac{2k-1}{2J}\pi\right)2^m\left(\frac{r}{2^m}\right)\right)
	\cos\left(\left(\frac{2k-1}{2J}\pi\right)2^m\left(\frac{s}{2^m}\right)\right)
\right.\nonumber\\
&\qquad+\left.
	\sin\left(\left(\frac{2k-1}{2J}\pi\right)2^m\left(\frac{r}{2^m}\right)\right)
	\sin\left(\left(\frac{2k-1}{2J}\pi\right)2^m\left(\frac{s}{2^m}\right)\right)
\right)\nonumber\\
=& \varphi_{J,r}\left(\frac{s}{2^m}\right)\\
=& \left\{
	\begin{array}{ll}
		J, & \mbox{ if } s=r,\\
		\frac{\sin((s-r)\pi)}{2\sin\left(\frac{(s-r)\pi}{2J}\right)}=0, & \mbox{ otherwise.}
	\end{array}
\right.
\end{align*}
Next, let 
$$\chi(x) = \sum^J_{k=1}\left(
	c_k \cos\left(\frac{(2k-1)\pi}{2J}2^m x\right)
	+ d_k \sin\left(\frac{(2k-1)\pi}{2J}2^m x\right)
\right),$$ 
and $\chi\left(\frac{r}{2^m}\right)=1 $ for some constants $c_k$ and $d_k$. By a simple application of the Cauchy-Schwarz inequality, we get
\begin{align*}
1 = \chi^2\left(\frac{r}{2^m}\right)
={}& \left(\sum^J_{k=1}\left(
	c_k \cos\left(\frac{(2k-1)\pi}{2J}r\right)
	+ d_k \sin\left(\frac{(2k-1)\pi}{2J}r\right)
\right)\right)^2\\
\leq{} & \left( \sum^J_{k=1}(c^2_k+d^2_k)\right)
\left(
	\sum^J_{k=1}\left(
		\left(\cos\left(\frac{(2k-1)\pi}{2J}r\right)\right)^2
		+\left(\sin\left(\frac{(2k-1)\pi}{2J}r\right)\right)^2
	\right)
\right)\\
={}& J ||\chi||^2_2.
\end{align*}
The last equality follows from the orthonormality of set $\Gamma_{m,J}$ and since $c_k$ and $d_k$ are arbitrary, $||\chi||^2_2\geq \frac{1}{J}$ for any $\chi\in V_J$, such that $\chi\left(\frac{r}{2^m}\right)=1$. On the other hand, as 
$$\varphi_{J,r}\left(\frac{r}{2^m}\right)= \sum^J_{k=1}\left(
	\left(\cos\left(\frac{(2k-1)\pi}{2J}2^m \frac{r}{2^m}\right)\right)^2
	+\left(\sin\left(\frac{(2k-1)\pi}{2J}2^m \frac{r}{2^m}\right)\right)^2
\right)=J,$$
We know that 
\begin{equation*}
\left|\left|
	\frac{\varphi_{J,r}}{\varphi_{J,r}(2^{-m}r)}
\right|\right|^2_2
=\frac{<\varphi_{J,r},\varphi_{J,r}>}{J^2}=\frac{1}{J},
\end{equation*} 
which concludes the proof of (b). Statement (c) is true since  $\{\varphi_{J,r}|r=1-J,2-J,\ldots,J\}$ has the same number of elements as the spanning set of $V_J$; its elements are orthogonal and therefore independent of each other.

Part (d) follows from parts (a) and (c). For any $v\in V_J$, $v(\cdot) = \sum^J_{s=1-J}\mathcal{V}_s \varphi_{J,s}(\cdot)$ by (c), and from part (a), we have 
\begin{equation*}
<v,\varphi_{J,r}>
=\sum^J_{s=1-J}\mathcal{V}_s<\varphi_{J,s},\varphi_{J,r}>=\sum^J_{s=1-J}\mathcal{V}_s\varphi_{J,s}\left(\frac{r}{2^m}\right)=v\left(\frac{r}{2^m}\right).
\end{equation*} 
\end{proof}
\end{theorem} 

The space $V_J$ and the scaling functions $\{\varphi_{J,r}\}_{r=1-J,\ldots, J}$ are our approximation space and our basis functions, respectively. As a result, for any function $v$ in the $L^2((-2^{-m}J,2^{-m}J])$ space, its projection on $V_J$, denoted as $H_{V_J} v$, can be written in the form $\frac{1}{J}\sum^J_{r=1-J}<H_{V_J} v,\varphi_{J,r}>\varphi_{J,r} = \frac{1}{J}\sum^J_{r=1-J}<v,\varphi_{J,r}>\varphi_{J,r}$. 

\subsection{Quick SWIFT formula and coefficients}\label{quick_SWIFT_formula_and_coefficients}
Assume we wish to approximate a finite integral $\int_\mathbb{R}v(\varsigma)q(\varsigma)d\varsigma$, where $v$ is within $L^2((-2^{-m}J,2^{-m}J])$ and we have $\int_\mathbb{R}q(\varsigma)d\varsigma<\infty$. 
We shall approach this problem by replacing $v$ by $H_{V_J}v$. This gives us the following approximation: 
\begin{align}
\int_{\mathbb{R}}v(\varsigma)q(\varsigma) d\varsigma 
\approx& \int_{\mathbb{R}}q(\varsigma)\frac{1}{J}\sum^J_{r=1-J}<v,\varphi_{J,r}>\varphi_{J,r}(\varsigma)d\varsigma \nonumber\\
=& \int_\mathbb{R}q(\varsigma)
\frac{1}{J}\sum^J_{r=1-J}
\frac{2^m}{J}\int^{2^{-m}J}_{-2^{-m}J} v(\varrho)
\sum^J_{k_1=1}\cos\left(C_{k_1}(2^m\varrho-r)\right) d\varrho
\sum^J_{k_2=1}\cos\left(C_{k_2}(2^m\varsigma-r)\right) d\varsigma
\nonumber\\
=& \sum^J_{r=1-J}\int_\mathbb{R}q(\varsigma)
\frac{2^{\frac{m}{2}}}{J}\sum^J_{k_2=1}\cos\left(C_{k_2}(2^m\varsigma-r)\right) d\varsigma
\int^{2^{-m}J}_{-2^{-m}J} v(\varrho)
\frac{2^{\frac{m}{2}}}{J}\sum^J_{k_!=1}\cos\left(C_{k_2}(2^m\varrho-r)\right) d\varrho,\nonumber
\end{align}
which is the SWIFT formula, proposed in \cite{ortiz_gracia_oosterlee_2016}, with $C_k:=\frac{2k-1}{2J}\pi$.
In the above derivation, we only listed the dependency to dummy variable $\varsigma$ of the functions $v$ and $q$. 
In practice, $v$ and $q$ will depend on other variables, like for example, time.
We will put the additional dependency in our notation without further notice in the remainder of this article whenever it is necessary for the presentation.

\begin{remark} While the accuracy of the approximation depends on other properties of the functions $v$ and $q$ and shall be studied in the rest of this article, $v\in L^2((-2^{-m}J,2^{-m}J])$ and $q$ being integrable are the only conditions needed for the above approximation to be well-defined.  
\end{remark}

\begin{remark} We only define the approximation on the set $(-2^{-m}J,2^{-m}J]$ that centers around zero. For any function $v$ such that $\int^b_a(v(\varsigma))^2 d\varsigma<\infty$ for a finite range $(a,b]$, we need to perform a change of variables $\varsigma'=\varsigma - \frac{a+b}{2}$, for $\varsigma'\in (a-\frac{a+b}{2},b-\frac{a+b}{2}]$, let $v'(\varsigma')=v(\varsigma'+\frac{a+b}{2})=v(\varsigma)$. Then, we can pick $J$ and $m$ accordingly and perform the approximation on $v'$. With a slight abuse of notation, we assume that we perform this tedious step implicitly and apply the SWIFT formula with any finite range $[a,b]$. 
\end{remark}

In this article, we shall adopt the quick variant of the SWIFT formula proposed in \cite{maree_ortiz-gracia_oosterlee_2016}. Instead of replacing $v$ with $H_{V_J}v$, we approximate $v$ by $$v(x) \approx \frac{1}{J}\sum^{J}_{r=1-J}v\left(\frac{r}{2^m}\right)\varphi_{J,r}(x).$$ 
The reason behind this is left for the error section, but this gives an approximation for $\mathbb{E}^x_p[v(t_{p+1},X^\Delta_{t_{p+1}})]$, which we see in the discrete-time approximation of FBSDE,
\begin{align}
\mathbb{E}^x_p[v(t_{p+1},X^\Delta_{t_{p+1}})]  
\approx & 
\mathbb{E}^x_p\left[
	\frac{1}{J}
	\sum^{J}_{r=1-J} 
	v\left(t_{p+1},\frac{r}{2^m}\right) 
	\varphi_{J,r}(X^\Delta_{t_{p+1}})
\right] 
 = 
 \frac{1}{J}\sum^{J}_{r=1-J}v\left(t_{p+1},\frac{r}{2^m}\right)\mathbb{E}^x_p[\varphi_{J,r}(X^\Delta_{t
 _{p+1}})].
 \nonumber
\end{align}
The expectation $\mathbb{E}^x_p[\varphi_{J,r}(X^\Delta_{t_{p+1}})]$ can be computed by
\begin{align}
\mathbb{E}^x_p[\varphi_{J,r}(X^\Delta_{t_{p+1}})] & = 
\mathbb{E}^x_p\left[
	\sum^J_{k=1}\cos\left(
		2^m C_k X^\Delta_{t_{p+1}} - C_k r
	\right)
\right]\nonumber\\
& = \Re\left\{ 
	\sum^J_{k=1}\mathbb{E}^x_p\left[	
		\exp\left(
			\imath 2^m C_k (x +\mu(t_p,x) \Delta t +\sigma(t_p,x) \Delta \omega_{p+1})
		\right)
		\exp\left(-\imath C_k r	\right)
	\right]
\right\}\nonumber\\
& = \Re \left\{\sum^J_{k=1} 
	\exp\left(
		\imath  2^m C_k x
	\right)
	\Phi\left( t_p, x, 2^m C_k \right)
	\exp\left(-\imath C_k r \right)
\right\}, \label{characteristic_SWIFT_basis}
\end{align}
where the real part of a complex number is denoted by $\Re\{\}$ and $\Phi$ is the characteristic function of $X^\Delta_{t_{p+1}}-X^\Delta_{t_p}$, i.e.
\begin{equation*}
\Phi(\tau,\varsigma,\varrho):=
\exp\left(
	\imath \varrho  \mu(\tau,\varsigma) \Delta t
	-\frac{1}{2} \varrho^2 \sigma^2(\tau, \varsigma) \Delta t
\right).
\end{equation*}
The authors of \cite{maree_ortiz-gracia_oosterlee_2016} demonstrated how to calculate the vector $(\mathbb{E}^x_p[\varphi_{J,r}(X^\Delta_{t_{p+1}})])_{r=(1-J,\ldots, J)}$ with a Fast Fourier Transform (FFT) algorithm. 
The computations induced by Equation \eqref{characteristic_SWIFT_basis} in our algorithm have the computation complexity of ${\rm O}(J\log(J))$.

Expectations in the form $\mathbb{E}^x_p[v(X^\Delta_{t_{p+1}})\Delta \omega_{p+1}]$ also occur in the discrete-time approximation of the FBSDE and they can be computed by:
\begin{align}
\mathbb{E}^x_p[v(t_{p+1},X^\Delta_{t_{p+1}})\Delta \omega_{p+1}] 
\approx &
\mathbb{E}^x_p\left[
	\frac{1}{J}
	\sum^{J}_{r=1-J}
	v\left(t_{p+1},\frac{r}{2^m}\right)
	\varphi_{J,r}(X^\Delta_{t_{p+1}})
	\Delta \omega_{p+1}
\right] \nonumber\\
= &
\frac{1}{J}\sum^{J}_{r=1-J}
v\left(t_{p+1},\frac{r}{2^m}\right)
\mathbb{E}^x_p[\varphi_{J,r}(X^\Delta_{t_{p+1}})\Delta \omega_{p+1}],
\nonumber
\end{align}
with $\mathbb{E}^x_p[\varphi_{J,r}(X^\Delta_{t_{p+1}})\Delta \omega_{p+1}]$ given by the formula:
\begin{align}
&
\mathbb{E}^x_p[\varphi_{J,r}(X^\Delta_{t_{p+1}})\Delta \omega_{p+1}]
= 
\sigma(t_p,x) \Delta t \mathbb{E}^x_p[D_x\varphi_{J,r}(X^\Delta_{t_{p+1}})]\nonumber\\
= & 
\Re\left\{
	\imath \sigma(t_p,x) \Delta t \sum^J_{k=1} 2^m C_k \mathbb{E}^x_p\left[
		\exp\left(
			\imath 2^m C_k
			(x + \mu(t_p,x) \Delta t + \sigma(t_p,x) \Delta \omega_{p+1})
		\right)
		\exp\left(-\imath C_k r\right)
	\right]
\right\}\nonumber\\ 
= &
\Re\left\{
	\imath \sigma(t_p,x) \Delta t 
	\sum^J_{k=1} 2^m C_k 
	\exp\left(\imath 2^m C_k x\right)
	\Phi\left(t_p, x, 2^m C_k\right)
	\exp\left(-\imath C_k r \right)
\right\}, \label{characteristic_SWIFT_Brownian}
\end{align}
where the first equality sign follows from an integration by parts argument and we also note that $
D_x\varphi_{J,r}(x) =  
-\sum^J_{k=1} 2^m C_k \sin \left( 
	2^m C_k x - C_k r
\right).$ 
Once again, we can use the FFT to compute these expectations.

In the next two sections, we will combine the quick variant of the SWIFT formula and the discrete-time approximation of the FBSDE in Equation \eqref{discretized}. 

\subsection{Quick SWIFT approximation of function $z^\Delta_p(x)$}\label{quick_SWIFT_approximation_of_function_z}

There are three different expectations, $\mathbb{E}^x_p[Z^\Delta_{t_{p+1}}]$, $\mathbb{E}^x_p[Y^\Delta_{p+1}\Delta\omega_{t_{p+1}}]$, and $\mathbb{E}^x_p[f(t_{p+1},{\bf X}^\Delta_{t_{p+1}})\Delta\omega_{p+1}]$, that need to be approximated in Equation \eqref{discrtized_z}. Applying the quick SWIFT formula, we have: 
\begin{align}
%
%
\mathbb{E}^x_p[Z^\Delta_{t_{p+1}}]
\approx &
\frac{1}{J}\sum^{J}_{r=1-J}
z^\Delta_{p+1}\left(\frac{r}{2^m}\right)
\Re \left\{\sum^J_{k=1} 
	e^{\imath  2^m C_k x}
	\Phi\left(t_p,x, 2^m C_k \right)
	e^{-\imath C_k r}
\right\};\nonumber\\
%
%
\mathbb{E}^x_p[Y^\Delta_{t_{p+1}}\Delta\omega_{p+1}]
\approx &
\frac{1}{J}\sum^{J}_{r=1-J}
y^\Delta_{p+1}\left(\frac{r}{2^m}\right)
\Re\left\{
	\imath \sigma(t_p,x) \Delta t 
	\sum^J_{k=1} 2^m C_k 
	e^{\imath 2^m C_k x}
	\Phi\left(t_p, x, 2^m C_k\right)
	e^{-\imath C_k r}
\right\};\nonumber\\
\mathbb{E}^x_p[f(t_{p+1},{\bf X}^\Delta_{t_{p+1}}) \Delta \omega_{p+1}]
\approx & 
\frac{1}{J} \sum^{J}_{r=1-J}
f\left(
	t_{p+1},
	\frac{r}{2^m},
	y^\Delta_{p+1}\left(\frac{r}{2^m}\right),
	z^\Delta_{p+1}\left(\frac{r}{2^m}\right)
\right)\cdot\nonumber\\
&
\Re\left\{
	\imath \sigma(t_p,x) \Delta t 
	\sum^J_{k=1} 2^m C_k 
	e^{\imath 2^m C_k x}
	\Phi\left(t_p, x, 2^m C_k\right)
	e^{-\imath C_k r}
\right\}\nonumber.
\end{align}

We denote the approximation of the FBSDE by a SWIFT-type formula combining with the Euler discretization at point $(t_p,x)$ by $(\hat{y}^\Delta_p(x),\hat{z}^\Delta_p(x))$, then $\hat{z}^\Delta_p$ satisfies the following recursive relation: 
\begin{align}
\hat{z}^\Delta_p(x)
= & \;
\Re \left\{
	\frac{1}{J}
	\sum^J_{k=1}
	e^{\imath  2^m C_k x}
	\left[ 
		\Phi\left(t_p,x, 2^m C_k \right)
		\sum^{J}_{r=1-J}
		\left(
			-\frac{1-\theta_2}{\theta_2}
			\hat{z}^\Delta_{p+1}\left(\frac{r}{2^m}\right)	
			e^{-\imath C_k r}
		\right)
	\right]
\right\}\nonumber\\
+ & \Re\left\{
	\frac{\imath}{J}\sigma(t_p,x)
	\sum^J_{k=1}
	e^{\imath 2^m C_k x}
	2^m C_k	\Phi\left(t_p, x, 2^m C_k\right)\cdot
\right.\nonumber\\
&\qquad\left.
	\left[
		\sum^{J}_{r=1-J}\left(
			\frac{1}{\theta_2}\hat{y}^\Delta_{p+1}\left(\frac{r}{2^m}\right)
			+\frac{(1-\theta_2)\Delta t}{\theta_2}
			f\left(
				t_{p+1},
				\frac{r}{2^m},
				\hat{y}^\Delta_{p+1}\left(\frac{r}{2^m}\right),
				\hat{z}^\Delta_{p+1}\left(\frac{r}{2^m}\right)
			\right)
			e^{-\imath C_k r}
		\right)
	\right]
\right\}, \label{function_z}
\end{align}
for $p = 0 ,1, \ldots, P-1$.

\subsection{Quick SWIFT approximation of function $y^\Delta_p(x)$}\label{quick_SWIFT_approximations_of_function_y}

Equation \eqref{discrtized_y} for $Y^\Delta_{t_p}$ contains an explicit and an implicit part if $\theta_1>0$. The explicit part is denoted by:
\begin{equation}
h (t_p, x) 
:= 
\mathbb{E}^x_p[Y^\Delta_{t_{p+1}}]
+\Delta t (1-\theta_1) \mathbb{E}^x_p[f(t_{p+1},{\bf X}^\Delta_{t_{p+1}})].\label{function_h}
\end{equation}
The function $h$ is a linear combination of two expectations, $\mathbb{E}^x_p[Y^\Delta_{t_{p+1}}]$ and $\mathbb{E}^x_p[f(t_{p+1},{\bf X}^\Delta_{t_{p+1}})]$, and they can be approximated by the following quick SWIFT formulas:
\begin{subequations} 
\begin{align}
\mathbb{E}^x_p[Y^\Delta_{t_{p+1}}]
\approx & 
\frac{1}{J}\sum^{J}_{r=1-J}y^\Delta_{p+1}\left(\frac{r}{2^m}\right)
\Re \left\{\sum^J_{k=1} 
	e^{\imath  2^m C_k x}
	\Phi\left(t_p, x, 2^m C_k \right)
	e^{-\imath C_k r}
\right\};\label{initial_guess}\\
\mathbb{E}^x_p[f(t_{p+1},{\bf X}^\Delta_{t_{p+1}})] 
\approx &
\frac{1}{J}\sum^{J}_{r=1-J}
f\left(
	t_{p+1},
	\frac{r}{2^m},
	y^\Delta_{p+1}\left(\frac{r}{2^m}\right),
	z^\Delta_{p+1}\left(\frac{r}{2^m}\right)
\right)
\Re \left\{\sum^J_{k=1} 
	e^{\imath  2^m C_k x}
	\Phi\left(t_p, x, 2^m C_k \right)
	e^{-\imath C_k r}
\right\}.
\end{align}
\end{subequations}
Therefore, we have an approximation for $h$:
\begin{align}
\hat{h}(t_p,x)
:= &
\mathbb{E}^x_p[\hat{y}^\Delta_{p+1}(X^\Delta_{t_{p+1}})]
+\Delta t (1-\theta_1) \mathbb{E}^x_{p}[f(t_{p+1},X^\Delta_{t_{p+1}},\hat{y}^\Delta_{p+1}(X^\Delta_{t_{p+1}}),\hat{z}^\Delta_{p+1}(X^\Delta_{t_{p+1}}))]
\nonumber\\
 = & \Re \left\{
	\frac{1}{J}\sum^J_{k=1}
	e^{\imath 2^m C_k x}
		\Phi\left(t_p, x, 2^m C_k \right)\cdot
\right.\nonumber\\
& \quad
\left.
		\sum^{J}_{r=1-J}
		\left[
			\hat{y}^\Delta_{p+1}\left(\frac{r}{2^m}\right)
			+ \Delta t (1 - \theta_1) 
			f\left(
				t_{p+1},
				\frac{r}{2^m},
				\hat{y}^\Delta_{p+1}\left(\frac{r}{2^m}\right),
				\hat{z}^\Delta_{p+1}\left(\frac{r}{2^m}\right)
			\right)
		\right]
		e^{-\imath C_k r}
\right\},\label{eq:approximation_function_h}
\end{align}
and the function $\hat{y}^\Delta_p$ is defined implicitly by:
\begin{equation}
\hat{y}^\Delta_p(x) =  \Delta t \theta_1 f(t_p,x,\hat{y}^\Delta_p(x),\hat{z}^\Delta_p(x)) + \hat{h}(t_p,x). \label{function_y}
\end{equation}

Whenever $\theta_1\neq 0$, {\it Picard iterations} are performed $I$ times to recover $\hat{y}^\Delta_p(x)$, which is the same procedure used in both \cite{gobet_lemor_warin_2005} and \cite{ruijter_oosterlee_2015}. 
The initial guess for the iterations is defined as the approximation of $\mathbb{E}^x_p[Y^\Delta_{t_{p+1}}]$ as in Equation \eqref{initial_guess}.
The conditions for convergent iterations and an extra error induced will be discussed in Section \ref{picard_iterations_error}.

The overall algorithm to generate an approximation $(\hat{y}^\Delta_0(x),\hat{z}^\Delta_0(x))$ for $(Y_0,Z_0)$ has been summarized in Algorithm \ref{quick_SWIFT_method}.

\begin{algorithm}[h]
\Begin{
	\For{$s = 1-J$ to $J$}{
		$\hat{y}^\Delta_P(2^{-m} s) = g(2^{-m} s)$, 
		$\hat{z}^\Delta_P(2^{-m} s) = \sigma D_x g(2^{-m} s)$ and \\ $\hat{f}^\Delta_P (2^{-m} s) = f(T,2^{-m} s, \hat{y}^\Delta_P(2^{-m} s),\hat{z}^\Delta_P(2^{-m} s)).$ 
	}
	Compute $(\mathbb{E}^{2^{-m}r}_{P-1}[\varphi_{J,k}(X^\Delta_{t_P})])_{r,k=1-J,\ldots,J}$ and $(\mathbb{E}^{2^{-m} r}_{P-1}[\varphi_{J,k}(X^\Delta_{t_P})\Delta\omega_P])_{r,k=1-J,\ldots,J}$ with (\ref{characteristic_SWIFT_basis}) and (\ref{characteristic_SWIFT_Brownian}).\\
	\For{$p=P-1$ to 1}{
		Compute the function $(\hat{z}^\Delta_p(2^{-m}s))_{s=1-J,\ldots, J}$ with (\ref{function_z}).\\
		Compute the function $(\hat{y}^\Delta_p(2^{-m} s))_{s=1-J,\dots, J}$ with (\ref{function_y}) and Picard iterations if necessary.\\
		Compute the functions $(f(t_p,2^{-m}s,\hat{y}^\Delta_p(2^{-m}s),\hat{z}^\Delta_p(2^{-m}s)))_{s=1-J,\ldots,J}$.\\
		Compute $(\mathbb{E}^{2^{-m}r}_{p-1}[\varphi_{J,k}(X^\Delta_{t_p})])_{r,k=1-J,\ldots,J}$ and $(\mathbb{E}^{2^{-m} r}_{p-1}[\varphi_{J,k}(X^\Delta_{t_p})\Delta\omega_p])_{r,k=1-J,\ldots,J}$ with (\ref{characteristic_SWIFT_basis}) and (\ref{characteristic_SWIFT_Brownian}) if the distribution of $(X^\Delta_{t_p}-X^\Delta_{t_{p-1}})$ is time-dependent.\\
	}
	Compute $\hat{z}^\Delta_0(x_0)$ and $\hat{y}^\Delta_0(x_0)$.
}	
\caption{Quick SWIFT method.}\label{quick_SWIFT_method}
\end{algorithm}

\section{Errors and computational complexity}\label{section:errors_and_computational_complexity}
In this section, we shall discuss the major components of the error when solving an FBSDE with a SWIFT-type method. They are the discretization error of the FBSDE, the error of approximation with the SWIFT formula and the error introduced by the Picard iteration. We will also discuss the computational complexity of the SWIFT method. For notational simplicity, we shall use $M$ to denote a generic constant whose value and dependency may change from line to line.

\subsection{Discretization error of the FBSDE}\label{discretization_error_of_the_BSDE}
The error due to the discrete-time approximation of the stochastic process depends on the parameters $\theta_1$ and $\theta_2$, the drift $\mu$, the  volatility $\sigma$, the driver function $f$ and the terminal condition $g$. It is difficult to provide a universal result for all FBSDEs for which our method can be applied. However, under some specific assumptions, we can derive an error bound for the global error due to time discretization. Adopting the following error notation:
\begin{align}
\varepsilon^y_p(X_p)
:= &
y_p(X_{t_p}) - y^\Delta_p(X^\Delta_{t_p}),
\;
\varepsilon^z_p
:=
z_p(X_{t_p}) - z^\Delta_p(X^\Delta_{t_p}),
\;
\varepsilon^f_p(X_{t_p})
:=
f(t_p,{\bf X}_{t_p}) - f(t_p, {\bf X}^\Delta_{t_p}),\nonumber
\end{align}
one of the results for the discretization error is  in the following theorem.
\begin{theorem}[\cite{ruijter_oosterlee_2015},Theorem 1.]\label{discretization_convergence}
Assume that the forward process has constant coefficients $\mu$ and $\sigma$ and the discretization scheme with $\theta_1=\theta_2=\frac{1}{2}$. If 
\begin{equation*}
\mathbb{E}^x_{P-1}[|\varepsilon^z_{P}|]\sim\mathcal{O}((\Delta t)^3),\;
\mathbb{E}^x_{P-1}[|\varepsilon^y_P|]\sim\mathcal{O}((\Delta t)^3),
\end{equation*}
then
\begin{equation*}
\mathbb{E}^x_0\left[|\varepsilon^y_p|+\sqrt{\Delta t}|\varepsilon^z_p|\right]\leq M(\Delta t)^2, \mbox{ for } 1\leq p \leq P,
\end{equation*}
where the constant $M$ depends on numbers $T$, $\mu$ and $\sigma$ and functions $g$ and $f$.
\end{theorem}
For general drift and diffusion coefficients, we may only have {\it first-order weak convergence} for the Euler discretization of forward process $X$ and it may become the dominating error of our algorithm. For the proof of Theorem \ref{discretization_convergence} and other convergence results, readers are referred to \cite{ruijter_oosterlee_2015} and the references therein. 
 
\subsection{Error in SWIFT formulas}\label{error_in_SWIFT_formulas}
In this subsection, we shall discuss the error of the numerical approximation by the SWIFT-type formulas. In order for our formula to be applicable for both the one-step approximation and the recursive case, we adopt the following setting. 

Consider an expectation $\mathbb{E}[v(X^\Delta_{t+\Delta t})|X_t^\Delta = x]$ for a function $v$ defined on $\mathbb{R}$ and assume that $v$ is continuous with all its left- and right-side derivatives well-defined, we define our expectation approximation as
\begin{equation}
\hat{\mathbb{E}}[v(X^\Delta_{t+\Delta t})|X^{\Delta}_t = x]:=\frac{1}{J}\sum^{J}_{r=1-J}\rho_v\left(\frac{r}{2^m}\right)\mathbb{E}[\varphi_{J,r}(X^\Delta_{t+\Delta t})|X^\Delta_t=x],\label{eq:approximated_expectation}
\end{equation}
where $\{\rho_v(2^{-m}r)\}_{r=1-J,\ldots,J}$ is an  approximation vector related to function $v$, to be defined.
For any function $v:\mathbb{R}\rightarrow\mathbb{R}$ and given range $(-2^{-m}J,2^{-m}J]$, we associate $v$ with {\it an alternating extension} defined below. 
\begin{definition}
An {\it alternating extension of a function $v$}, denoted by $\tilde{v}$, is a function defined on $\mathbb{R}$ such that it satisfies:
\begin{enumerate}
\item[(a)] $\tilde{v}(x) = v(x) \quad\forall x \in (-2^{-m}J,2^{-m}J];$
\item[(b)] $\tilde{v}(x + 2^{1-m}J) = -\tilde{v}(x) 
			 \quad \forall x \in \mathbb{R}.$ 
\end{enumerate}
\end{definition}
The difference between the approximation value and the true value is given by 
\begin{align}
&
\mathbb{E}[v(X^\Delta_{t+\Delta t})|X_t^\Delta = x] - \hat{\mathbb{E}}[v(X^\Delta_{t+\Delta t})|X^{\Delta}_t = x]\nonumber\\
=&
\mathbb{E}[v(X^\Delta_{t+\Delta t})|X^\Delta_t = x] - \mathbb{E}[\tilde{v}(X^\Delta_{t+\Delta t})|X^\Delta_t = x]
+ \mathbb{E}[\tilde{v}(X^\Delta_{t+\Delta t})|X^\Delta_t = x] - \mathbb{E}[H_{V_J}v(X^\Delta_{t+\Delta t})|X^\Delta_t = x] \nonumber\\
+ &
 \mathbb{E}\left[
	\left.
		\frac{1}{J}\sum^J_{r=1-J}
		<H_{V_J} v, \varphi_{J,r}>
		\varphi_{J,r}(X^\Delta_{t+\Delta t})
	\right|X^\Delta_t = x
\right]
- \frac{1}{J}\sum^{J}_{r=1-J}\rho_v\left(\frac{r}{2^m}\right)\mathbb{E}[\varphi_{J,r}(X^\Delta_{t+\Delta t})|X^\Delta_t=x]\nonumber\\
= &
\mathbb{E}[v(X^\Delta_{t+\Delta t}){\bf 1}_{\{X^\Delta_{t+\Delta t }\notin (-2^{-m}J,2^{-m}J]\}}|X^\Delta_t = x]
- \mathbb{E}[\tilde{v}(X^\Delta_{t+\Delta t}){\bf 1}_{\{X^\Delta_{t+\Delta t }\notin (-2^{-m}J,2^{-m}J]\}}|X^\Delta_t = x]\nonumber\\
+ & 
 \mathbb{E}[\tilde{v}(X^\Delta_{t+\Delta t})- H_{V_J}v(X^\Delta_{t+\Delta t})|X^\Delta_t = x] 
+ \frac{1}{J}\sum^J_{r=1-J}\left(
	H_{V_J} v \left(\frac{r}{2^m}\right)
	- \rho_v\left(\frac{r}{2^m}\right)
\right)\mathbb{E}[\varphi_{J,r}(X^\Delta_{t+\Delta t})|X^\Delta_t=x]
\label{SWIFT_error_formula}\\
= & 
\mathbb{E}[v(X^\Delta_{t+\Delta t}){\bf 1}_{\{X^\Delta_{t+\Delta t }\notin (-2^{-m}J,2^{-m}J]\}}|X^\Delta_t = x] 
- \mathbb{E}[\tilde{v}(X^\Delta_{t+\Delta t}){\bf 1}_{\{X^\Delta_{t+\Delta t }\notin (-2^{-m}J,2^{-m}J]\}}|X^\Delta_t = x]\nonumber\\
+ &
 \mathbb{E}[\tilde{v}(X^\Delta_{t+\Delta t})- H_{V_J}v(X^\Delta_{t+\Delta t})|X^\Delta_t = x] 
+ 
 \frac{1}{J}\sum^J_{r=1-J}\left(
	H_{V_J} v \left(\frac{r}{2^m}\right)
	- v\left(\frac{r}{2^m}\right)
\right)\mathbb{E}[\varphi_{J,r}(X^\Delta_{t+\Delta t})|X^\Delta_t=x]\nonumber\\
+ &
 \frac{1}{J}\sum^J_{r=1-J}\left(
	v \left(\frac{r}{2^m}\right)
	- \rho_v\left(\frac{r}{2^m}\right)
\right)\mathbb{E}[\varphi_{J,r}(X^\Delta_{t+\Delta t})|X^\Delta_t=x].
\label{error_formula}
\end{align}
We derive the above formula by telescoping and using the properties of the scaling function. 
By taking absolute values on both sides of Equation (\ref{error_formula}), we get a simple bound for the approximation error:
\begin{align}
&
\left|\mathbb{E}[v(X^\Delta_{t+\Delta t})|X^\Delta_t = x] - \hat{\mathbb{E}}[v(X^\Delta_{t+\Delta t})|X^{\Delta}_t = x]\right|\nonumber\\
\leq & \;
|\mathbb{E}[v(X^\Delta_{t+\Delta t}){\bf 1}_{\{X^\Delta_{t+\Delta t }\notin (-2^{-m}J,2^{-m}J]\}}|X^\Delta_t = x]|
+ \mathbb{E}[|\tilde{v}(X^\Delta_{t+\Delta t})|{\bf 1}_{\{X^\Delta_{t+\Delta t }\notin (-2^{-m}J,2^{-m}J]\}}|X^\Delta_t = x]\nonumber\\
+ & 
 \mathbb{E}[|\tilde{v}(X^\Delta_{t+\Delta t})- H_{V_J}v(X^\Delta_{t+\Delta t})|\,|X^\Delta_t = x] 
+ 
 \frac{1}{J}\sum^J_{r=1-J}\left|
	H_{V_J} v \left(\frac{r}{2^m}\right)
	- \tilde{v}\left(\frac{r}{2^m}\right)
\right|\,|\mathbb{E}[\varphi_{J,r}(X^\Delta_{t+\Delta t})|X^\Delta_t=x]|\nonumber\\
+ &
 \frac{1}{J}\sum^J_{r=1-J}\left|
	v \left(\frac{r}{2^m}\right)
	- \rho_v\left(\frac{r}{2^m}\right)
\right|\,|\mathbb{E}[\varphi_{J,r}(X^\Delta_{t+\Delta t})|X^\Delta_t=x]|.\nonumber
\end{align}

Errors of a SWIFT-type method can be separated into four (in Equation (\ref{SWIFT_error_formula})) or five (in Equation (\ref{error_formula})) parts and they will be discussed one by one.
Note that the first three terms in Equation (\ref{SWIFT_error_formula}) and (\ref{error_formula}) are the same.

The first error term is related to the tail behaviour of the probability measure and function $v$. 
It is finite, as otherwise the original expectation would be infinite. 
Also, its value should decrease (heuristically speaking, not in strict mathematical sense) when a wider computational domain is used. 
Assuming $v$ to be uniformly bounded by a number $M$, this term is bounded by $M\cdot\mathbb{P}(X^\Delta_{t+\Delta t }\notin (-2^{-m}J,2^{-m}J]|X^\Delta_t = x)$. 
Similarly, the second term is related to the tail probability. 
Because of the continuity of $v$ and the periodicity of $\tilde{v}$, $\tilde{v}$ is uniformly bounded by some number $M'$ and the second error term is bounded by $M'\cdot\mathbb{P}(X^\Delta_{t+\Delta t }\notin (-2^{-m}J,2^{-m}J]|X^\Delta_t = x)$.

The third part is related to the projection error on $V_J$. From the assumption that $v$ is continuous with all left- and right-side derivatives of $v$ existing, $\tilde{v}=\lim_{J\rightarrow \infty} H_{V_J}v$, a.e.. This can be shown by adopting the classical Dirichlet's kernel argument. Applying the dominated convergence theorem,
\begin{align*}
\mathbb{E}[\tilde{v}(X^\Delta_{t+\Delta t})- H_{V_J}v(X^\Delta_{t+\Delta t})\,|X^\Delta_t = x]
= & 
\sum^\infty_{k=J+1}
<v,\cos(2^m C_j\cdot)> 
\mathbb{E}[\cos(2^mC_jX^\Delta_{t+\Delta t})|X^\Delta_t =x]\nonumber\\
+ &
\sum^\infty_{k=J+1}
<v,\sin(2^m C_j\cdot)> 
\mathbb{E}[\sin(2^mC_jX^\Delta_{t+\Delta t})|X^\Delta_t =x].
\end{align*} 
Note that in this part, we only require $\tilde{v}=\lim_{J\rightarrow \infty} H_{V_J}v$, a.e., so that we can relax the requirement on $v$ from being continuous to being piecewise continuous. Using an integration by parts argument, if the forward process has a smooth density, this error converges exponentially with respect to $J$ but increases with the computational range. 

\begin{remark}
In fact, the projection error in the SWIFT formula can be controlled under alternative conditions. Assume that the probability density function $q$ of $X^\Delta_{t+\Delta t}|X^\Delta_t =x$, is in $L^2(\mathbb{R})$, then, 
\begin{align}
&
 |\mathbb{E}[\tilde{v}(X^\Delta_{t+\Delta t})- H_{V_J}v(X^\Delta_{t+\Delta t})|X^\Delta_t = x]
- \mathbb{E}[\tilde{v}(X^\Delta_{t+\Delta t}){\bf 1}_{\{X^\Delta_{t+\Delta t }\notin (-2^{-m}J,2^{-m}J]\}}|X^\Delta_t = x]|\nonumber\\
= & 
|\mathbb{E}[\tilde{v}(X^\Delta_{t+\Delta t})- H_{V_J}v(X^\Delta_{t+\Delta t}){\bf 1}_{\{X^\Delta_{t+\Delta t }\in (-2^{-m}J,2^{-m}J]\}}|X^\Delta_t = x]
\nonumber\\
- &
\mathbb{E}[H_{V_J}v(X^\Delta_{t+\Delta t}){\bf 1}_{\{X^\Delta_{t+\Delta t }\notin (-2^{-m}J,2^{-m}J]\}}|X^\Delta_t = x]|\nonumber\\
\leq &
||\tilde{v}- H_{V_J}v||_2
\left(\int_{\mathbb{R}}q^2(\varsigma|x) d\varsigma\right)^\frac{1}{2}
+ M_3 \mathbb{P}(X^\Delta_{t+\Delta t }\notin (-2^{-m}J,2^{-m}J]|X^\Delta_t = x),\nonumber
\end{align} 
with $M_3$ depending on the function $H_{V_J}v$. While we do not use this alternative proof in this article, it implies that the SWIFT formula can be used in a more general setting and is suitable for other applications as well.
\end{remark}

In the usual variant of the SWIFT formula, we set $\rho_v(2^{-m} r) = <v,\varphi_{J,r}>$, so the fourth term of Equation (\ref{SWIFT_error_formula}) is by definition zero and the error of applying the SWIFT formula will only consist of the first three terms. 
However, the calculation of $<v,\varphi_{J,r}>$ is difficult and time-consuming, if not impossible in practice, especially in the recursive case. 
Therefore, we propose picking $\rho_v(2^{-m} r)= \hat{v} (2^{-m} r)$, {\em an approximation of the original function $v$}. 
While it will introduce an extra error, we shall demonstrate that this error can be controlled and that the computation is much simpler. 

For the sum in the fourth term of Equation \eqref{error_formula}, we need to consider two cases.
When $r \neq J$, the pointwise convergence of $H_{V_J}v$ guarantees that 
$\left|
	H_{V_J} v \left(\frac{r}{2^m}\right)
	- \tilde{v}\left(\frac{r}{2^m}\right)
\right|\rightarrow 0$.
Therefore, these terms are bounded when $J$ is large enough.
When $r=J$, it is likely that $\tilde{v}$ is discontinuous at $2^{-m} J$ and the above argument does not hold.
However, we note that this error term is also a weighted sum of $H_{V_J}v(2^{-m} r) - v(2^{-m} r)$, with the weight given by $\frac{1}{J}\mathbb{E}[\varphi_{J,r}(X^\Delta_{t+\Delta t})|X^\Delta_t=x]$. 
Assume that $\mathbb{P}(X^\Delta_{t+\Delta t}\not\in (\lambda-b,\lambda+b)) < \epsilon_1$ and $\frac{1}{J}\varphi_{J,J}(x) < \epsilon_2$, when $x\in(\lambda-b,\lambda+b)$, for some positive numbers $b, \epsilon_1$ and $\epsilon_2$ and some number $\lambda$, then
\begin{align}
\frac{1}{J}\mathbb{E}[\varphi_{J,r}(X^\Delta_{t+\Delta t})|X^\Delta_t=x]
< & 
\mathbb{P}(X^\Delta_{t+\Delta t}\not\in (\lambda-b,\lambda+b))
+\epsilon_2 \mathbb{P}(X^\Delta_{t+\Delta t}\in (\lambda-b,\lambda+b))
< 
\epsilon_1 + \epsilon_2.\nonumber
\end{align}
These assumptions are satisfied when the distribution of $X^\Delta$ is centered around a point $\lambda$, which is true with a diffusion process whose diffusion coefficient is small, the computational range is sufficiently large  so that $\lambda$ is far away from the boundary, and the wavelet order is sufficiently high. 
If these conditions are met, the weight for the term $\left|
	H_{V_J} v \left(\frac{J}{2^m}\right)
	- \tilde{v}\left(\frac{J}{2^m}\right)
\right|$ is small and the weighted term can be bounded.
By combining the two arguments above, we can bound this error term when the computational range is sufficiently large and the wavelet order is sufficiently high.

In the one-step case, we pick $\rho_{v}=v$. Equation \eqref{SWIFT_error_formula} along with the above analysis covers all approximation errors. 

In the backward recursive case, we can use Equation \eqref{error_formula} to study the error propagation at each time step. 
For example, in our BSDE algorithm, we let $v = y^\Delta_{p+1}$ or $z^\Delta_{p+1}$ and $\rho_{v}= \hat{y}^\Delta_{p+1}$ or $\hat{z}^\Delta_{p+1}$.
In these cases, the fifth term of Equation \eqref{error_formula} is comparing our approximation with the true value at the next time step.
The error accumulates in a recursive way by applying this error analysis from time step $t_0$ to $t_{p-1}$.
Further discussion of the error propagation will be given in Section \ref{errors of the recursive scheme}.

The error for approximating $\mathbb{E}[v(X^\Delta_{t+\Delta t})\Delta\omega_{p+1}|X^\Delta_t=x]$ with 
\begin{equation}
\hat{\mathbb{E}}[v(X^\Delta_{t+\Delta t})(\omega_{t+\Delta_t}-\omega_t)|X^{\Delta}_t = x]
:=
\frac{1}{J}\sum^{J}_{r=1-J}\rho_v\left(\frac{r}{2^m}\right)
\mathbb{E}[\varphi_{J,r}(X^\Delta_{t+\Delta t})(\omega_{t+\Delta}-\omega_t)|X^\Delta_t=x],
\label{eq:approximated_expectation_with_omega}
\end{equation}
can be studied in a similar way.

\begin{remark}
It is clear from the derivation that the assumption of the function $v$ being continuous with left- and right-side derivatives is crucial in applying the quick version of the SWIFT formula. 
In our FBSDE algorithm, the functions $y$ and $z$ at intermediate time points, $p=1,\ldots, P-1$, satisfy the above conditions. 
This can be observed from Equations (\ref{function_z}) and (\ref{function_y}). 
However, we may still face an issue at the terminal time, as $D_x g$ may contain discontinuities. 
{\it We propose a mixed algorithm to deal with this situation.} 
At the terminal time, the usual SWIFT formulas are used and then the algorithm switches to the quick version in all subsequent time steps, see Algorithm~2.
\end{remark}

\begin{algorithm}[h]
\Begin{
	\For{$s = 1-J$ to $J$}{
		$\hat{y}^\Delta_{P}(2^{-m} s) = <g,\varphi_{J,s}>$, 
		$\hat{z}^\Delta_P(2^{-m} s) = <\sigma D_x g, \varphi
		_{J,s}>$ and \\ $\hat{f}^\Delta_P (2^{-m} s) = <f(T,\cdot, g(\cdot),\sigma D_x(\cdot)),\varphi_{J,r}>.$ 
	}
	Compute $(\mathbb{E}^{2^{-m}r}_{P-1}[\varphi_{J,k}(X^\Delta_{t_P})])_{r,k=1-J,\ldots,J}$ and $(\mathbb{E}^{2^{-m} r}_{P-1}[\varphi_{J,k}(X^\Delta_{t_P})\Delta\omega_P])_{r,k=1-J,\ldots,J}$ with \eqref{characteristic_SWIFT_basis} and \eqref{characteristic_SWIFT_Brownian}.\\
	\For{$p=P-1$ to 1}{
		Compute the function $(\hat{z}^\Delta_p(2^{-m}s))_{s=1-J,\ldots, J}$ with (\ref{function_z}).\\
		Compute the function $(\hat{y}^\Delta_p(2^{-m} s))_{s=1-J,\dots, J}$ with (\ref{function_y}) and Picard iterations if necessary.\\
		Compute the functions $(f(t_p,2^{-m}s,\hat{y}^\Delta_p(2^{-m}s),\hat{z}^\Delta_p(2^{-m}s)))_{s=1-J,\ldots,J}$.\\
		Compute $(\mathbb{E}^{2^{-m}r}_{p-1}[\varphi_{J,k}(X^\Delta_{t_p})])_{r,k=1-J,\ldots,J}$ and $(\mathbb{E}^{2^{-m} r}_{p-1}[\varphi_{J,k}(X^\Delta_{t_p})\Delta\omega_p])_{r,k=1-J,\ldots,J}$ with (\ref{characteristic_SWIFT_basis}) and (\ref{characteristic_SWIFT_Brownian}) if the distribution of $(X^\Delta_{t_p}-X^\Delta_{t_{p-1}})$ is time-dependent.\\
	}
	Compute $\hat{z}^\Delta_0(x_0)$ and $\hat{y}^\Delta_0(x_0)$.
}	
\caption{Mixed quick SWIFT method.}
\end{algorithm}

\subsection{Picard iteration error}
\label{picard_iterations_error}

When $\theta_1\neq 0$, a Picard iteration must be performed with the equation:
\begin{equation*}
y = \Delta t \theta_1 f(t_p,x,y,\hat{z}^\Delta_p(x)) + \hat{h}(t_p,x),
\end{equation*} 
to find the fixed point $y$. It is well-known that this iterative algorithm will converge to the unique fixed point if the function $\Delta t \theta_1 f$ is a contraction map of $y$, namely, 
\begin{equation*}
|\Delta t \theta_1 f(t_p,x,y_1,\hat{z}^\Delta_p(x))-\Delta t \theta_1 f(t_p,x,y_2,\hat{z}^\Delta_p(x))|\leq \xi|y_1-y_2|,
\end{equation*}
with $\xi\in[0,1)$ for all $x\in(-2^{-m}J,2^{-m}J]$. This condition is satisfied when the driver function is Lipschitz in $y$ and $\Delta t$ is small enough. 

We adopt the following notation:
\begin{equation*}
\left\{
	\begin{array}{l}
		\hat{y}^{\Delta,I}_{P}(x) := g(x);\\
		\displaystyle \hat{y}^{\Delta,0}_p(x) := 
		\frac{1}{J}\sum^J_{r=1-J}\hat{y}^{\Delta,I}_{p+1}\left(\frac{r}{2^m}\right)
		\Re \left\{\sum^J_{k=1}e^{\imath 2^m C_k x}\Phi(t_p,x,2^mC_k)e^{-\imath C_k r}\right\};\\
		\hat{y}^{\Delta,i+1}_p(x) :=
		 \Delta t \theta_1 f(t_p, x, \hat{y}^{\Delta,i}_p(x), \hat{z}^\Delta_p (x)) + \hat{h}(t_p,x),  
	\end{array}
\right.
\end{equation*}
for $p = 0, \ldots, P-1$ and $i = 0, \ldots, I-1$.
It is clear that $\hat{y}^{\Delta,I}_p(x) = \hat{h}(t_p,x)$ when $\theta_1 = 0$ and $I\geq 1$, which is the explicit scheme.  
The above notations are consistent with the notations in Section \ref{quick_SWIFT_approximations_of_function_y} except we should replace the $\hat{y}^\Delta_{p+1}$ with $\hat{y}^{\Delta,I}_{p+1}$ in equation \eqref{function_h}. 
Furthermore, for any given $x$, we know by definition that $y^\Delta_p(x)$ is the unique fixed point that satisfies 
\begin{equation*}
	y = \Delta t \theta_1 f(t_p, x,y, z^{\Delta}_p(x))+h(t_p,x).
\end{equation*}
Note that the notation $\hat{y}^\Delta_p$ was defined by Equation \eqref{function_y}.
   
With the above notations and given the extra information that $f$ is Lipschitz with respect to $z$, we can derive the one-step approximation error for function $y^\Delta$:
\begin{align*}
&
|\hat{y}^{\Delta, I}_p(x) - y^\Delta_p(x)| 
\leq |\hat{y}^{\Delta,I}_p(x)-\hat{y}^{\Delta}_p(x)| + |\hat{y}^\Delta_p(x) - y^\Delta_p(x)|\nonumber\\ 
\leq &
\varepsilon^{Picard}_p+\Delta t \theta_1 |f(t_p,x,\hat{y}^\Delta_p(x),\hat{z}^\Delta_p(x))-f(t_p,x,y^\Delta_p(x), z^\Delta_p(x))|+|\hat{h}(t_p,x)-h(t_p,x)|\nonumber\\
\leq &
\varepsilon^{Picard}_p+\xi |\hat{y}^\Delta_p(x)-y^\Delta_p(x)|
+ \xi|\hat{z}^\Delta_p(x)-z^\Delta_p(x)|+|\hat{h}(t_p,x)-h(t_p,x)|\nonumber\\
\leq &
(1+\xi)\varepsilon^{Picard}_p+\xi|\hat{y}^{\Delta,I}_p(x)-y^\Delta_p(x)|
+ \xi|\hat{z}^\Delta_p(x)-z^\Delta_p(x)|+|\hat{h}(t_p,x)-h(t_p,x)|.
\end{align*}
The term $\varepsilon^{Picard}_p:=|\hat{y}^{\Delta, I}_p(x) - \hat{y}^\Delta_p(x)|$ is the error of applying Picard iterations, which depends on $\Delta t$ and the Lipschitz coefficient of $f$ with respect to $y$, as stated before in this section. 
The constant $M$ is related to the Lipschitz coefficient of $f$ from standing assumption (A3) and $\xi:=M\Delta t \theta_1 \leq 1$.
The last inequality is due to a telescoping argument of the term $ |\hat{y}^\Delta_p(x)-y^\Delta_p(x)|$.
Rearranging the terms gives us the following error bound: 
\begin{equation}
|\hat{y}^{\Delta, I}_p(x) - y^\Delta_p(x)| 
\leq\frac{1+\xi}{1-\xi}\varepsilon^{Picard}_p+\frac{1}{1-\xi}(\xi|\hat{z}^\Delta_p(x)-z^\Delta_p(x)|+|\hat{h}(t_p,x)-h(t_p,x)|).
\label{picard_iteration_error}
\end{equation}

\subsection{The errors of the FBSDE recursive scheme}
\label{errors of the recursive scheme} 

Given an FBSDE system, a time partition and a discretization scheme,  we may apply the result in Section \ref{error_in_SWIFT_formulas} to derive a local approximation error for the related expectations. 
When we are  approximating expectations with the functions $y^\Delta_{p+1}$ and $z^\Delta_{p+1}$ in the BSDE scheme, the approximation vector is given by 
\begin{equation*}
\left\{
	\begin{array}{l}
		\rho_{y^\Delta_{p+1}} = \{\hat{y}^{\Delta,I}_{p+1}(2^{-m}r)\}_{r=1-J,\ldots, J};\\
		\rho_{z^\Delta_{p+1}} = \{\hat{z}^\Delta_{p+1}(2^{-m}r)\}_{r=1-J,\ldots, J},
	\end{array}
\right.
\end{equation*}
for $p= 0, \ldots, P-2$. 
At the terminal time $t_P=T$, they are defined as
\begin{equation*}
\left\{
	\begin{array}{l}
		\rho_{y^\Delta_P} = \{Y^\Delta_T|X^\Delta_T=2^{-m}r\}_{r=1-J,\ldots, J};\\
		\rho_{z^\Delta_P} = \{Z^\Delta_T|X^\Delta_T=2^{-m}r\}_{r=1-J,\ldots, J},
	\end{array}
\right.
\end{equation*}
or 
\begin{equation*}
\left\{
	\begin{array}{l}
		\rho_{y^\Delta_P} = \{<Y^\Delta_T,\varphi_{J,r}>\}_{r=1-J,\ldots, J};\\
		\rho_{z^\Delta_P} = \{<Z^\Delta_T,\varphi_{J,r}>\}_{r=1-J,\ldots, J},
	\end{array}
\right.
\end{equation*}
depending on the scheme we used.
When approximating expectations involving  $f(t_{p+1},x, y^{\Delta}_{p+1}(x), z^\Delta_{p+1}(x))$, the approximation vector $\rho_{f_{p+1}} = \{f(t_{p+1},2^{-m}r, \rho_{{y}^\Delta_{p+1}}(2^{-m}r), \rho_{z^\Delta_{p+1}}(2^{-m}r)\}_{r=1-J,\ldots,J}$, for $p = 0, \ldots P-1$.

From Equation \eqref{error_formula}, we know that the approximation error for the SWIFT formula consists of four parts. 
Therefore, the local approximation errors at point $(t,x)$ by the SWIFT formula for any function $v$, denoted as $\zeta^{m,J}_{v}(t_p,x)$ or $\zeta^{m,J,\omega}_{v}(t_p,x)$ for the two types of expectation, are given by  
\begin{align}
\zeta^{m,J}_{v}(t_p,x) :=  &
\mathbb{E}^x_p[v(X^\Delta_{t_{p+1}}){\bf 1}_{\{X^\Delta_{t_{p+1}}\notin (-2^{-m}J,2^{-m}J]\}}]
- \mathbb{E}^x_p[\tilde{v}(X^\Delta_{t_{p+1}}){\bf 1}_{\{X^\Delta_{t_{p+1}}\notin (-2^{-m}J,2^{-m}J]\}}]\nonumber\\
+ & 
\mathbb{E}^x_p[\tilde{v}(X^\Delta_{t_{p+1}})- H_{V_J}v(X^\Delta_{t_{p+1}})] 
+ \frac{1}{J}\sum^J_{r=1-J}\left(
	H_{V_J} v \left(\frac{r}{2^m}\right)
	- v\left(\frac{r}{2^m}\right)
\right)\mathbb{E}^x_p[\varphi_{J,r}(X^\Delta_{t_{p+1}})];\nonumber\\
\zeta^{m,J,\omega}_{v}(t_p,x) : =  &
\mathbb{E}^x_p[v(X^\Delta_{t_{p+1}}){\bf 1}_{\{X^\Delta_{t_{p+1}}\notin (-2^{-m}J,2^{-m}J]\}}\Delta \omega_{p+1}] 
- \mathbb{E}^x_p[\tilde{v}(X^\Delta_{t_{p+1}}){\bf 1}_{\{X^\Delta_{t_{p+1}}\notin (-2^{-m}J,2^{-m}J]\}}\Delta\omega_{p+1}]\nonumber\\
+ & 
\mathbb{E}^x_p[(\tilde{v}(X^\Delta_{t_{p+1}})- H_{V_J}v(X^\Delta_{t_{p+1}}))\Delta\omega_{p+1}] 
+ \frac{1}{J}\sum^J_{r=1-J}\left(
	H_{V_J} v \left(\frac{r}{2^m}\right)
	- v\left(\frac{r}{2^m}\right)
\right)\mathbb{E}^x_p[\varphi_{J,r}(X^\Delta_{t_{p+1}})\Delta\omega_{p+1}].\nonumber
\end{align}

Applying all the results above and the standing assumptions, we can derive the recurring error formulas of our SWIFT BSDE scheme:
\begin{align}
\frac{1}{M_1}|z^\Delta_p(x)-\hat{z}^\Delta_p(x)|
\leq & 
|\zeta^{m,J}_{z^\Delta_{p+1}}(t_p,x)|
+|\zeta^{m,J,\omega}_{y^\Delta_{p+1}}(t_p,x)|
+|\zeta^{m,J,\omega}_{f_{p+1}}(t_p,x)|
\nonumber\\
+ &
\frac{1}{J}\sum^J_{r=1-J}\left|
	z^\Delta_{p+1} \left(\frac{r}{2^m}\right)
	- \rho_{z^\Delta_{p+1}}\left(\frac{r}{2^m}\right)
\right|(|\mathbb{E}^x_p[\varphi_{J,r}(X^\Delta_{t_{p+1}})]|+\mathbb{E}^x_p[\varphi_{J,r}(X^\Delta_{t_{p+1}})\Delta\omega_{p+1}]|)\nonumber\\
+ &
\frac{1}{J}\sum^J_{r=1-J}\left|
	y^\Delta_{p+1} \left(\frac{r}{2^m}\right)
	- \rho_{y^\Delta_{p+1}}\left(\frac{r}{2^m}\right)
\right|\,|\mathbb{E}^x_p[\varphi_{J,r}(X^\Delta_{t_{p+1}})\Delta\omega_{p+1}]|,\label{recurring_z}
\end{align}
and 
\begin{align}
\frac{1}{M_2}|y^\Delta_p(x)-\hat{y}^{\Delta,I}_p(x)|
\leq & 
M_3
+ |\zeta^{m,J}_{z^\Delta_{p+!}}(t_p,x)|
+|\zeta^{m,J,\omega}_{y^\Delta_{p+1}}(t_p,x)|
+|\zeta^{m,J,\omega}_{f_{p+1}}(t_p,x)|
+|\zeta^{m,J}_{y^\Delta_{p+1}}(t_p,x)|
+|\zeta^{m,J}_{f_{p+1}}(t_p,x)|\nonumber\\
+ &
\frac{1}{J}\sum^J_{r=1-J}\left|
	z^\Delta_{p+1} \left(\frac{r}{2^m}\right)
	- \hat{z}^\Delta_{p+1}\left(\frac{r}{2^m}\right)
\right|(|\mathbb{E}^x_p[\varphi_{J,r}(X^\Delta_{t_{p+1}})]|+\mathbb{E}^x_p[\varphi_{J,r}(X^\Delta_{t_{p+1}})\Delta\omega_{p+1}]|)\nonumber\\
+ &
\frac{1}{J}\sum^J_{r=1-J}\left|
	y^\Delta_{p+1} \left(\frac{r}{2^m}\right)
	- \hat{y}^\Delta_{p+1}\left(\frac{r}{2^m}\right)
\right|(|\mathbb{E}^x_p[\varphi_{J,r}(X^\Delta_{t_{p+1}})]|+\mathbb{E}^x_p[\varphi_{J,r}(X^\Delta_{t_{p+1}})\Delta\omega_{p+1}]|),\label{recurring_y}
\end{align}
with constants $M_1$, $M_2$ and $M_3$ depending on the underlying FBSDE, the discretization scheme and the Picard iteration errors, but not depending on $m$ and $J$.
Their proof is left for the appendix.

\subsubsection{Error bound and choice of parameters}
\label{SWIFT_parameters}

An error bound at $(0, x_0)$ for applying the SWIFT scheme to a FBSDE system can be found by repeatedly applying Equations \eqref{recurring_z} and \eqref{recurring_y}.
This bound is given by the weighted sum of the local approximation errors for each point in the grid $\{(0,x_0)\}\cup\{(t_p,2^{-m}r)|t=1,\ldots, P \mbox{ and } r=1-J,\ldots, J\}$, with the weight being $1$ for $(0,x_0)$ and the weight being
\begin{equation}
\sum_{u\in\upsilon_p}\frac{1}{J^p}\prod^p_{l=1}(|\mathbb{E}^{u_{l-1}}_l[\varphi_{J,u_l}(X^\Delta_{t_l})]|+|\mathbb{E}^{u_{l-1}}_l[\varphi_{J,u_{l}}(X^\Delta_{t_l})\Delta\omega_{l+1}]|), 
\label{eq: weight_of_error_bound}
\end{equation}
for $\{(t_p,2^{-m}r)|t=1,\ldots, P \mbox{ and } r=1-J,\ldots, J\}$, where $\upsilon_p$ is the set containing length $p+1$ vectors $u=(u_0, u_1,\ldots, u_p)$, with the first element $u_0=x_0$ and other elements equal to $2^{-m}r$ for $r=1-J,\ldots, J$. 
A simple example of deriving such an error bound is included in the appendix.

However, since this error bound uses local approximation errors from multiple points in a grid, actually calculating the error bound would be costly.
The weight and the local error behave differently at different grid points.
Whenever $|r|$ is small, the local error converges exponentially in numerical tests with fixed $2^{-m}J$ and increasing $J$, but it may fail to converge to zero when $|r|$ is close to $J$.
On the other hand, when $|r|$ is close to $J$, the weight in Equation \eqref{eq: weight_of_error_bound} tends to zero quickly and reduces the error.
We do not have a simple formula to describe this balance.
Last but not least, parameter $P$, the number of time points, affects the total number of terms in the error bound, the value of the local error and the value of the weight in Equation \eqref{eq: weight_of_error_bound}.
It is highly non-trivial to quantify the effect of $P$ on the overall error from this error bound. 

In practice, we would recommend a three-step approach to select the parameters for the scheme and get a global picture of what the error may be.
First of all, pick parameter $P$ based on the discretization error for $(X,Y,Z)$.
This can be done either through the error bound in Section \ref{discretization_error_of_the_BSDE} or other existing error bounds in the literature.
The reason is that $m$ and $J$ have no effect on the discretization error while $P$ affects each part of the approximation error.
Next, we should choose our parameters $J$ and $m$ according to error formula \eqref{error_formula}. 
The interest is in the third term in the equation, which increases with the truncation range but decreases with the wavelet order $J$. 
Therefore we should first fix the truncation range $2^{-m} J$ to a constant value $a$ in our scheme, the tail probability outside our computational range is below a fixed tolerance level.
This can be done heuristically by considering the cumulants of $X_T$, see \cite{fang_oosterlee_2009_original_COS}.
Finally, we pick a $J$ value such that the overall error converges and adjust $m$ accordingly so that the truncation range remains the same.
This approach is very useful for applications (compared to the error bound itself).   

\subsection{Computational Complexity} 

At each time-step $t_p$, the following operations have to be performed:
\begin{itemize}
\item {\em Computation} of $\mathbb{E}^x_p[\varphi_{J,k}(X^\Delta_{t_{p+1}})]$ and $\mathbb{E}^x_p[\varphi_{J,k}(X^\Delta_{t_{p+1}})\Delta\omega_{p+1}]$ by the FFT algorithm, in ${\rm O}(J\log(J))$ operations. It is calculated only once at the terminal time-step if the characteristic function of $X^\Delta$ does not depend on the time point;

\item {\em Computation} of $\hat{z}^\Delta_p(x)$, $\hat{h}(t_p,x)$ and $\hat{y}^{\Delta,0}_p(x)$ by matrix-vector multiplications, in $\mathcal{O}(J^2)$ operations;

\item {\em Computation} of $\hat{y}^{\Delta,I}$ by I Picard iterations on an $x$-grid, in $\mathcal{O}(I J)$ operations;

\item {\em Evaluation} of $f(t_p,x,\hat{y}^{\Delta,I}_p(x), \hat{z}^{\Delta}_p(x))$ in $\mathcal{O}(J)$ operations.
\end{itemize}
The most time-consuming part in our algorithm is the matrix-vector multiplication. The proposed algorithms have linear computational complexity with respect to the time-steps $P$ and the starting evaluation at terminal time is of order $\mathcal{O}(J)$. In total, the complexity of the SWIFT-type methods is $\mathcal{O}(J+P[J^2+ I J +J\log(J) +J])$.

\section{Antireflective boundary}\label{section:anti-reflective_boundary}

Recalling Equation \eqref{error_formula}, the approximation of $\mathbb{E}[v(X^\Delta_{t+\Delta t})|X_t^\Delta = x]$ by the SWIFT formula may have a significant local error when two conditions are satisfied. 
The first condition being that the alternating extension $\tilde{v}$ diverges from $v$ in the range $[-\eta-2^{-m}J,-2^{-m}J]$ or $[2^{-m}J,\eta+2^{-m}J]$, for some number $\eta>0$ and the second condition being that the probability of $X^\Delta_{t+\Delta t}|X^\Delta_t=x$ in the aforementioned ranges is large. 
While the first condition is almost always true, given that $X^\Delta$ is a diffusion process, the second condition is true only when the starting point $x$ is close to the boundary $-2^{-m}J$ or $2^{-m}J$. 
Therefore, there may be intervals $(-2^{-m}J,\alpha)$ and $(\beta,2^{-m}J]$ where the SWIFT formula is inaccurate.

We propose using an antireflective boundary technique to deal with this issue. Antireflective boundary conditions form a popular technique for extrapolation in image deblurring methods. For its applications in image deblurring, the reader may refer to \cite{donatelli_serra-capizzano_2010} and the references therein. 

\begin{figure}[t!]
\centering
\setlength{\unitlength}{1.5cm}
\begin{picture}(10,1)(0,0)
\thicklines
\put(1,0.5){\vector(1,0){8}}
\thicklines
\put(9,0.5){\vector(-1,0){8}}
\put(2,0.4){\line(0,1){0.2}}
\put(2.8,0.4){\line(0,1){0.2}}
\put(3.6,0.4){\line(0,1){0.2}}
\put(6.4,0.4){\line(0,1){0.2}}
\put(7.2,0.4){\line(0,1){0.2}}
\put(8,0.4){\line(0,1){0.2}}
\put(1.6,0.1){$-2^{-m}J$}
\put(2.7,0.7){$\alpha$}
\put(3.2,0.1){$2\alpha+2^{-m}J$}
\put(6,0.1){$2\beta-2^{-m}J$}
\put(7.7,0.1){$2^{-m}J$}
\put(7.1,0.7){$\beta$}
\end{picture}
\caption{The approximation range $(-2^{-m}J, 2^{-m}J]$ and the accuracy range $[\alpha,\beta]$.}
\label{figure:anti-reflective_range}
\end{figure}
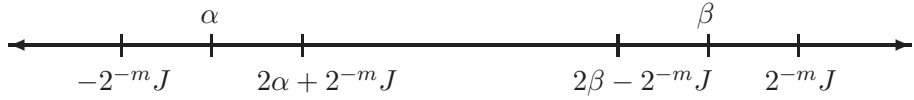

In practice, assume that we approximate a function $\vartheta(x)$ by $\hat{\vartheta}(x)$ on $(-2^{-m}J,2^{-m}J]$ and we know that the approximation is accurate for $x\in[\alpha,\beta]$, namely, $|\vartheta(x)-\hat{\vartheta}(x)|<\epsilon$, for some small positive real number $\epsilon$. 
Given the numbers $\alpha>-2^{-m}J$ and $\beta <2^{-m}J$, so that there is some inaccuracy near the boundary but $(\alpha,2\alpha+2^{-m}J),(2\beta - 2^{-m}J,\beta) \subset [\alpha,\beta]$ (see Figure \ref{figure:anti-reflective_range}). 
We would extrapolate an approximation of $\vartheta(x)$ for $x\in (-2^{-m}J,\alpha)$ and $x\in (\beta,2^{-m}J]$ by applying antireflective conditions with the accurate approximation. 
For $x\in(-2^{-m}J,2^{-m}J]$, we define
\begin{equation}
\left\{
	\begin{array}{ll}
		\hat{\vartheta}^a(x) :=
		2 \hat{\vartheta}(\alpha) -  
		\hat{\vartheta}(2\alpha-x) &
		\mbox{ for } x\in(-2^{-m}J,\alpha);\\
		\hat{\vartheta}^a(x) :=
		\hat{\vartheta}(x) &
		\mbox{ for } x\in[\alpha,\beta];\\
		\hat{\vartheta}^a(x) :=
		2\hat{\vartheta}(\beta) 
		- \hat{\vartheta}(2\beta - x)&
		\mbox{ for } x\in(\beta, 2^{-m}J), 
	\end{array}
\right.
\label{anti-reflective_formula}
\end{equation}
and use $\hat{\vartheta}^a$ instead of $\hat{\vartheta}$ as our approximation.

If $\vartheta$ is two times continuously differentiable on $\mathbb{R}$, then, by a simple application of Taylor's theorem, we have:
\begin{align}
\vartheta(x) 
= & 
\vartheta(\alpha) 
+ \frac{d\vartheta}{dx}(\alpha)(x-\alpha) 
+ \frac{1}{2}\frac{d^2\vartheta}{dx^2}(\varsigma_1)(x-\alpha)^2;\nonumber\\
\vartheta(2\alpha - x)
= &
\vartheta(\alpha)
- \frac{d\vartheta}{dx}(\alpha)(x-\alpha)
+ \frac{1}{2}\frac{d^2\vartheta}{dx^2}(\varsigma_2)(x-\alpha)^2,  \nonumber
\end{align}
where $x\in (2^{-m}J,\alpha)$, $\varsigma_1\in (x,\alpha)$ and $\varsigma_2 \in (\alpha, 2\alpha-x)$. 
The approximation error for $x\in(-2^{-m}J,\alpha)$ is then bounded by 
\begin{align}
|\vartheta(x)-\hat{\vartheta}^a(x)|
\leq &
|\vartheta(x) - 2\vartheta(\alpha) + \vartheta(2\alpha-x)|
+2|\vartheta(\alpha)-\hat{\vartheta}(\alpha)|
+|\vartheta(2\alpha - x)-\hat{\vartheta}(2\alpha-x)|\nonumber\\
\leq & 
{(-2^{-m}J-\alpha)^2}\sup_{\varsigma\in(-2^{-m}J,2\alpha+2^{-m}J)}\left|\frac{d^2\vartheta}{dx^2}(\varsigma)\right|
+3 \epsilon.\nonumber
\end{align}
A similar formula can be derived for the set $(2\beta-2^{-m}J,2^{-m}J)$. For the recursive scheme, one can just apply Equation \eqref{anti-reflective_formula} at each time-step.

\begin{remark} 
The range of accurate approximations for the SWIFT formula depends on the distribution of $X^\Delta_{t+\Delta t}|_{X^\Delta_t=x}$, and, therefore, it is model dependent. 
The performance of applying the antireflective boundary technique with the SWIFT formula depends on the smoothness of the target function with respect to starting point $x$, the accuracy of the SWIFT formula in the range $[\alpha,\beta]$ and the length of the range. 
\end{remark} 

\section{Numerical experiments}\label{section:numerical_experiments}

Several numerical studies have been performed in MATLAB 9.0.0. The computer is equipped with Intel(R) Core(TM) i5-2520M CPU @ 2.50GHz and 7.7 GB RAM. Four different discretization schemes have been used to test the effect of various choices of $\theta$ on the numerical algorithm. They are:\\ \\
\begin{tabular}[l]{ l l p{.1cm} l l}
Scheme A: & $\theta_1=0$, $\theta_2=1$, & & 
Scheme C: & $\theta_1=1$, $\theta_2=1$, \\
Scheme B: & $\theta_1=0.5$, $\theta_2=1$, & &
Scheme D: & $\theta_1=0.5$, $\theta_2=0.5$.
\end{tabular}\\ \\
The function $z^\Delta_p$ is solved explicitly in all schemes while ${y}^\Delta_p$ is solved explicitly for scheme A and implicitly with  5 Picard iterations for the other schemes. 

For each given example, we associate our numerical algorithm with the computational domain $[\kappa_1-L\sqrt{\kappa_2},\kappa_1+L\sqrt{\kappa_2}]$, where cumulants $\kappa_1 = x_0 +\mu(0,x_0) T$ and $\kappa_2 = \sigma^2(0,x_0) T$ and $L =10$.   
It is similar to the setting in \cite{ruijter_oosterlee_2015}.  The value $2^{-m} J = L\sqrt{\kappa_2}$ is a constant for each example. The value of $J$ is assumed to be $2^9$.

\subsection{Example 1}
This example has been studied in \cite{ruijter_oosterlee_2015} and is originally from \cite{zhao_li_zhang_2012}. The considered FBSDE is 
\begin{equation}
\left\{
\begin{array}{l}
dX_t =  d\omega_t,\\
dY_t = -(Y_tZ_t-Z_t+2.5Y_t-\sin(t+X_t)\cos(t+X_t)-2\sin(t+X_t))dt+Z_td\omega_t.
\end{array}
\right.
\end{equation}
We take the initial and terminal conditions $x_0 = 0$ and $Y_T = \sin(X_T+T)$.

The exact solution is given by 
\begin{equation}
(Y_t,Z_t)=(\sin(X_t+t),\cos(X_t+t)).
\end{equation}
The terminal time is set to be $T=1$ and $(Y_0,Z_0)=(0,1)$. The driver function $f$ depends on both time $t$ and current state $X_t$. 
The results for the quick SWIFT method are presented in Figure \ref{quick_SWIFT_example1} while the results for the {\em mixed SWIFT method} are presented in Figure \ref{mixed_quick_SWIFT_example1}. 
We observe that there are no significant differences between the quick SWIFT and mixed SWIFT method. For schemes A, B and C, both approximation errors for $Y_0(x_0)$ and $Z_0(x_0)$ are of ${\rm O}(\Delta t)$ order, while the errors converge with ${\rm O}((\Delta t)^2)$ for scheme D.

\begin{remark}
The driver functions for some examples in this section are not universally Lipschitz.
However, one should notice that for the contraction argument in Section \ref{picard_iterations_error} to be valid, for any fixed $z$, the driver function should be Lipschitz with respect to $y$.
All the driver functions in this section satisfy this weaker condition.
We aim for clear presentation rather than general applicability when we designed the assumptions and conditions and our method can be applied to a broader class of BSDEs than we described here.
For a more in-depth analysis of the application of Picard iterations in the numerical treatment of BSDEs, the reader is referred to \cite{gobet_lemor_warin_2005}. 
\end{remark}

\subsection{Example 2: European call option}
\label{section_example_2}

Next, we calculate the price $v(t,S_T)$ of a call option under the Black-Scholes model by solving an FBSDE. The underlying process satisfies:
\begin{equation}
dS_t = \bar{\mu} S_t dt + \bar{\sigma} S_t d\omega_t.
\end{equation}
Along the line of reasoning in \cite{ruijter_oosterlee_2015}, we assume that the financial market is complete, there are no trading restrictions and the call option can be perfectly hedged. Then $(v(t,S_t),\bar{\sigma} S_t D_s v(t,S_t))$ solves the FBSDE,
\begin{equation}
\left\{
\begin{array}{l}
dS_t =  \bar{\mu} S_t dt + \bar{\sigma} S_t d\omega_t,\\
dY_t = -(-rY_t-\frac{\bar{\mu}-r}{\bar{\sigma}}Z_t)dt+Z_td\omega_t,
\end{array}
\right.
\end{equation}
with terminal condition $Y_T=\max(S_T-K,0)$. The driver function is continuous and linear with respect to $y$ and $z$. We use the following parameter values in our test: 
\begin{equation}
S_0=100, K=100, r=0.1, \bar{\mu}=0.2, \bar{\sigma}=0.25,T=0.1.
\end{equation}
The exact solutions $Y_0 = 3.65997$ and $Z_0=14.14823$ are given by the Black-Scholes formula first shown in \cite{black_scholes_1973}. We switch to the log-asset domain $X_t = \log(S_t)$ and solve 
\begin{equation}
\left\{
\begin{array}{l}
dX_t = \left(\bar{\mu}-\frac{1}{2}\bar{\sigma}^2\right)dt +\bar{\sigma}d\omega_t,\\
dY_t = -(-rY_t-\frac{\bar{\mu}-r}{\bar{\sigma}}Z_t)dt+Z_td\omega_t,
\end{array}
\right.
\label{european_price}
\end{equation}
where $Y_T=\max(\exp(X_T)-K,0)$.

For the result of the quick SWIFT method, we refer to Figure \ref{quick_SWIFT_example2}. Immediately, we notice that the result of scheme D does not improve when increasing the number of time-steps. This is due to the discontinuity of the terminal value of $Z$ which creates a significant error. For the mixed SWIFT method, shown in Figure~\ref{mixed_quick_SWIFT_example2}, the error from the discontinuity  has been removed. The approximate values $\hat{y}^\Delta_0(x_0)$ and $\hat{z}^\Delta_0(x_0)$ converge with approximately order 1 with respect to $\Delta t$ for schemes A, B, and C and about order 2 for scheme D.

Since the driver function in Equation (\ref{european_price}) depends on $\bar{\mu}$, the approximation error also depends on $\bar{\mu}$, even though the final result $v(0,x_0)$ is unrelated to the drift. For the same number of time-steps $P$, the error increased with the increase of $\bar{\mu}$, as shown in Figure \ref{qexample2_mu}. 

The approximation algorithm can be further improved by applying antireflective boundary conditions in the recursive time-steps. 
In Figure \ref{qexample2_anti-relective} we see results of adding an antireflective step in a mixed SWIFT algorithm. 
The approximations near the middle of computational range are almost identical with the reference value, but the approximations with antireflective adjustment near both ends of the interval appear to be much better than the ones without.

\subsection{Example 3: Bid-ask spread for interest rate}
We next consider a financial model introduced in \cite{bergman_1995}, where we have distinct borrowing and leading interest rates. The resulting market is imperfect and the driver function is non-linear.

Suppose that an agent can invest in bonds with risk-free return rate $r$ and borrow money at rate $R>r$. 
For any European-type derivative with payoff $g(X_T)$ at time $T$, where the underlying asset $S_t=\log(X_t)$ follows a geometric Brownian motion, its price at time 0 can be obtained by solving the FBSDE: 
\begin{equation*}
\left\{
\begin{array}{l}
dX_t = \left(\bar{\mu}-\frac{1}{2}\bar{\sigma}^2\right)dt +\bar{\sigma}d\omega_t,\\
dY_t =  -\left(-r Y_t-\frac{\bar{\mu}-r}{\bar{\sigma}}Z_t-(R-r) \min \left(Y_t-\frac{Z_t}{\bar{\sigma}},0\right)\right)dt+Z_td\omega_t,
\end{array}
\right.
\end{equation*}
with the payoff as the terminal condition. We use the example studied in both \cite{bender_stein_2012} and \cite{ruijter_oosterlee_2015}. The payoff function is given by 
\begin{align*}
g(X_T)=(e^{X_T}-K_1)^+-2(e^{X_T}-K_2)^+,
\end{align*}
which equals a combination of a long call with strike $K_1=95$ and two short calls with strike $K_2=105$. We use the parameter values 
\begin{equation*}
S_0=100, r=0.01, \bar{\mu}=0.05, \bar{\sigma} =0.2, T=0.25, K_1=95, K_2=105, R=0.06.
\end{equation*}

We notice that $\hat{z}^\Delta_0(x_0)$ fails to converge to the reference value with scheme D in Figure \ref{quick_SWIFT_example3a}. The reference values, $Y_0=2.9584544$ and $Z_0=0.55319$, are obtained by the BCOS method with a large number of time-steps $P$. Switching to the mixed SWIFT method, whose results are shown in Figure \ref{mixed_quick_SWIFT_example3a}, the approximated error for $Y$ converges to zero with order of about 1 for schemes A, B and C and converges with order $\frac{3}{2}$ for scheme D. For schemes B and C, we also have a first-order convergence for $Z$ but the convergence order is higher for schemes A and D. 

\subsection{Example 4}
This example is taken from \cite{zhao_fu_zhou_2014}. For the forward process, the drift and diffusion coefficients are time- and state-dependent. We aim to solve the following FBSDE: 
\begin{equation*}
\left\{
\begin{array}{l}
dX_t = \frac{1}{1+2\exp(t+X_t)}dt + \frac{\exp(t+X_t)}{1+\exp(t+X_t)} d\omega_t,\\
dY_t =  -\left(-\frac{2Y_t}{1+2\exp(t+X_t)}-\frac{1}{2}
\left(
	\frac{Y_tZ_t}{1+\exp(t+X_t)}
	-Y_t^2 Z_t	
\right)\right)dt+Z_td\omega_t,
\end{array}
\right.
\end{equation*}
with the terminal condition $Y_T=g(X_T)=\frac{\exp(T+X_T)}{1+\exp(T+X_T)}$.

The exact solutions are given by
\begin{equation}
(Y_t,Z_t) = 
\left(
	\frac{\exp(t+X_t)}{1+\exp(t+X_t)},
	\frac{(\exp(t+X_t))^2}{(1+\exp(t+X_t))^3}
\right).
\end{equation}
We choose terminal time $T=1$ and initial condition $x_0 = 1$.

For the results of the quick SWIFT method, we refer to Figure \ref{quick_SWIFT_example4}.
While the total error is different for each scheme, the approximated values  $\hat{y}^\Delta_0(x_0)$ and $\hat{z}^\Delta_0(x_0)$ converge with ${\rm O}(\Delta t)$ for all schemes, as expected. 
Here the weak order of the Euler scheme plays a prominent role as the drift and volatility are state- and time- dependent.

\subsection{Discussion}

Compared with the BCOS method, the computational time for the SWIFT-type method is slightly lower when the number of basis functions used is the same and the forward process is independent of time. 
The most time-consuming portion is the matrix-vector multiplication used to calculate the value of $\hat{z}^\Delta_p$ and $\hat{h}$. 
We acknowledge that for the same error range, the BCOS and SWIFT-type methods may require different numbers of basis functions. 

From the numerical experiments, we conclude that the computation with scheme D often fails to converge with $\Delta t$ when the time-step is small when using the quick SWIFT method. 
This is due to the discontinuity of ${z}^\Delta_P(x)$ in $x$. 
Schemes A, B and C behave similarly for the quick SWIFT and mixed SWIFT methods in our examples. 
However, scheme D often has the best performance in the mixed SWIFT method. 
This means that damping the discontinuity in our scheme is beneficial. 
The graphs also demonstrate that with a proper choice of SWIFT parameters, the approximation error itself will be dominated by the discretization error and decreases with respect to the increase of parameter $P$. 
It implies that the error of calculating expectation with SWIFT is relatively small.  

\section{Conclusion}\label{section:conclusion}
A new probabilistic method for solving FBSDEs numerically has been proposed in this article. It is derived from a time-discretization of the forward and backward stochastic differential equations, taking conditional expectations to get an $\mathcal{F}_{t_p}$-adapted approximation and calculating the conditional expectations with the quick variant of the SWIFT formula. 

We have shown that in order to apply the quick variant of the SWIFT formula, the continuity of the target function has to be ensured. While applying the quick variant of SWIFT formula instead of the original version introduces an extra error, it is of the same order of the original version when the target function is continuous and drops quickly with respect to $J$, due to the exponential convergence of the characteristic function for a smooth density. The error of applying the SWIFT method is relatively minor compared to the discretization error for the stochastic process. However, the quick variant of the SWIFT formula can greatly reduce the difficulties of our algorithm and increase the computational speed. So we believe that the mixed SWIFT method provides a good balance between efficiency and accuracy.

We have discussed the different approximation errors in detail in the paper. Additional attention is needed for the error of the SWIFT formula near the computational boundary, as we explained in Section \ref{section:anti-reflective_boundary}. We also demonstrated how to improve our algorithm with the anti-reflective boundary conditions. Finally, the applicability and the effectiveness of our numerical algorithm have been tested with various FBSDEs, which all give positive results with the mixed SWIFT method. 

Overall, applying the SWIFT method to solve discretized BSDEs retains the high accuracy of Fourier inversion techniques, although the computations involved are greatly simplified. 
We also gain additional freedom in adjusting the approximation values at each time point.
  
\bibliographystyle{plain}
\bibliography{BSWIFT.bib} 

\appendix
\section{Appendix}

\subsection{Pointwise convergence of orthogonal projection}

We shall demonstrate that under some mild assumptions on a square integrable function in $(-2^{-m}J,2^{-m}J]$, its orthogonal projection on $V_J$ converges to the original function in a pointwise fashion, therefore bounding our approximation error. It is an adaptation of the standard Dirichlet kernel argument to our setting, a similar proof can be found in standard Fourier series textbook, like \cite{pinkus_zafrany_1997}. 

\begin{theorem}
Let $g$ be a square integrable function defined on the set $(-2^{-m}J,2^{-m}J]$ and the left- and right-side derivatives exist everywhere for its alternating extension $\tilde{g}$. If $\tilde{g}$ is continuous in a neighborhood around point $x$, the following result holds:
$$\lim_{J\rightarrow\infty}H_{V_J}g(x) = \tilde{g}(x).$$
\end{theorem}
\begin{proof}
By direct calculation,
\begin{align}
&
H_{V_J}g(x) \nonumber\\
= & 
\frac{2^m}{J} \sum^J_{k = 1}\left(
	\int^{2^{-m} J}_{-2^{-m} J} g(\varsigma) 
	\cos\left(2^m C_k \varsigma \right) d \varsigma 
	\cos\left(2^m C_k x\right) 
	+\int^{2^{-m} J}_{-2^{-m} J} g(\varsigma) 
	\sin\left(2^m C_k \varsigma \right) d \varsigma 
	\sin\left(2^m C_k x\right) 
\right)\nonumber\\
= & 
\frac{2^m}{J} \sum^J_{k = 1}\left(
	\int^{2^{-m} J}_{-2^{-m} J} \tilde{g}(\varsigma) 
	\cos\left(2^m C_k (\varsigma - x) \right) d \varsigma 
\right)
= 
\frac{2^m}{J} \sum^J_{k = 1}\left(
	\int^{2^{-m} J}_{-2^{-m} J} \tilde{g}(\varpi+x) 
	\cos\left(2^m C_k \varpi \right) d \varpi 
\right)\nonumber\\
= &
\frac{2^m}{J} 
\int^{2^{-m} J}_{-2^{-m} J} \tilde{g}(\varpi+x) 
\Upsilon_J (\varpi) d \varpi, 
\end{align}
where 
\begin{equation}
\Upsilon_J(x) := \sum^J_{k=1}\cos(2^m C_k x) = \frac{\sin(2^{m-1}\pi x)}{\sin(2^{m-1}\pi x/J)}.
\end{equation}
It can be shown that 
\begin{equation}
\frac{2^m}{J}\int^{2^{-m}J}_0 \Upsilon_J(\varsigma) d\varsigma = 
\frac{2^m}{J}\int_{-2^{-m}J}^0 \Upsilon_J(\varsigma) d\varsigma =
\frac{2}{\pi}\sum^J_{k=1}\frac{(-1)^{k+1}}{2k-1} =: \frac{2}{\pi} G_J.
\end{equation}
In fact, $G_J$ is the famous Gregory-Leibniz series and we have $\lim_{J\rightarrow\infty} G_J =\frac{\pi}{4}$.

Based on our assumption, at point $x$, $\lim_{h\rightarrow 0+}\tilde{g}(x+h)=\lim_{h\rightarrow 0-}\tilde{g}(x+h)=\tilde{g}(x)$. Also, $\frac{g(s+x)-g(x)}{s}\mathbbm{1}_{\{s\in(0,2^{-m}J)\}}$ and $\frac{g(s+x)-g(x)}{s}\mathbbm{1}_{\{s\in(-2^{-m}J,0)\}}$ are integrable functions on $(2^{-m}J,2^{-m}J]$. Note that in our construction, $2^{-m}J$ is a positive constant $(a)$ and $m$ is a function of $J$.

Therefore,
\begin{align}
&
H_{V_J}g(x) - \frac{4}{\pi} G_J\tilde{g}(x)\nonumber\\
= &
\frac{2^m}{J} 
\int^{a}_{0} (\tilde{g}(\varpi+x)-\tilde{g}(x)) 
\Upsilon_J (\varpi) d \varpi
+\frac{2^m}{J} 
\int^{0}_{-a} (\tilde{g}(\varpi+x)-\tilde{g}(x)) 
\Upsilon_J (\varpi) d \varpi\nonumber\\
= &
\frac{2}{\pi}\int^{a}_{0} \frac{\tilde{g}(\varpi+x)-\tilde{g}(x)}{\varpi}
\frac{\pi\varpi/2a}{\sin(\pi \varpi/2a)} 
\sin(2^{m-1}\pi\varpi) d \varpi\nonumber\\
&
+\frac{2}{\pi}\int^{0}_{-a} \frac{\tilde{g}(\varpi+x)-\tilde{g}(x)}{\varpi}
\frac{\pi\varpi/2a}{\sin(\pi \varpi/2a)} 
\sin(2^{m-1}\pi\varpi) d \varpi.
\end{align}
Since $\frac{x}{\sin(x)}$ is integrable on $[-\frac{\pi}{2}, \frac{\pi}{2}]$ and $m$ tends to infinity whenever $J$ tends to infinity, the last two terms go to 0, when $J$ tends to infinity. This is due to the Riemann-Lebesgue Lemma. So, we have 
\begin{equation}
\lim_{J\rightarrow\infty} H_{V_J}g(x) 
= \lim_{J\rightarrow\infty} \frac{4}{\pi}G_J\tilde{g}(x) = \tilde{g}(x), 
\end{equation}
and complete the proof.
\end{proof}

\subsection{Proof of Equations \eqref{recurring_z} and \eqref{recurring_y}}
\begin{proof}
Using Equations \eqref{discrtized_z}, \eqref{eq:approximated_expectation},  \eqref{eq:approximated_expectation_with_omega}, \eqref{error_formula}, and standing assumption (A3), we have
\begin{align*}
	& |\hat{z}^\Delta_p(x) - z^\Delta_p(x)| \\
	\leq &
	\frac{1-\theta_2}{\theta_2}|\mathbb{E}^x_p[z^\Delta_{p+1}(X^\Delta_{t_{p+1}})]- \hat{\mathbb{E}}[z^{\Delta}_{p+1}(X^\Delta_{t_{p+1}})|X^\Delta_{t_p}=x]|
	\\
	+ &
	\frac{1}{\theta_2\Delta t}|\mathbb{E}^x_p[y^\Delta_{p+1}(X^\Delta_{t_{p+1}})\Delta\omega_{p+1}]
	-\hat{\mathbb{E}}[y^\Delta_{p+1}(X^\Delta_{t_{p+1}})\Delta\omega_{p+1}|X^\Delta_{t_p}=x]|\\
	+ &
	\frac{1-\theta_2}{\theta_2}
	|\mathbb{E}^x_p[f(t_{p+1},X^\Delta_{t_{p+1}},y^\Delta_{p+1}(X^\Delta_{t_{p+1}}), z^\Delta_{p+1}(X^\Delta_{t_{p+1}}))\Delta\omega_{p+1}]\\
	& \hspace{30pt}
	-\hat{\mathbb{E}}[f(t_{p+1},X^\Delta_{t_{p+1}},y^\Delta_{p+1}(X^\Delta_{t_{p+1}}), z^\Delta_{p+1}(X^\Delta_{t_{p+1}}))\Delta\omega_{p+1}|X^\Delta_{t_p}=x]|\\
	\leq & 
	\frac{1-\theta_2}{\theta_2}|\zeta^{m,J}_{z^{\Delta}_{p+1}}(t_p,x)|
	+\frac{1}{\theta_2 \Delta t}|\zeta^{m,J,\omega}_{y^{\Delta}_{p+1}}(t_p,x)|
	+\frac{1-\theta_2}{\theta_2}|\zeta^{m,J,\omega}_{f_{p+1}}(t_p,x)|\\
	+ & \frac{1-\theta_2}{\theta_2}\frac{1}{J}\sum^J_{r=1-J}\left|z^\Delta_{p+1}\left(\frac{r}{2^m}\right)-\rho_{z^{\Delta}_{p+1}}\left(\frac{r}{2^m}\right)\right||\mathbb{E}^x_p[\varphi_{J,r}(X^\Delta_{t_{p+1}})]|\\
	+ & \frac{1}{\theta_2 \Delta t}\frac{1}{J}\sum^{J}_{r=1-J}\left|y^\Delta_{p+1}\left(\frac{r}{2^m}\right)-\rho_{y^\Delta_{p+1}}\left(\frac{r}{2^m}\right)\right||\mathbb{E}^x_p[\varphi_{J,r}(X^\Delta_{t_{p+1}})\Delta\omega_{p+1}]|\\
	+ & \frac{1-\theta_2}{\theta_2}\frac{1}{J}\sum^J_{r=1-J}\left|f\left(t_{p+1},\frac{r}{2^m}, y^\Delta_{p+1}\left(\frac{r}{2^m}\right), z^\Delta_{p+1}\left(\frac{r}{2^m}\right)\right)-\rho_{f_{p+1}}\left(\frac{r}{2^m}\right)\right||\mathbb{E}^x_p[\varphi_{J,r}(X^\Delta_{t_{p+1}})\Delta\omega_{p+1}]|\\
	\leq & 
	\frac{1-\theta_2}{\theta_2}|\zeta^{m,J}_{z^{\Delta}_{p+1}}(t_p,x)|
	+\frac{1}{\theta_2 \Delta t}|\zeta^{m,J,\omega}_{y^{\Delta}_{p+1}}(t_p,x)|
	+\frac{1-\theta_2}{\theta_2}|\zeta^{m,J,\omega}_{f_{p+1}}(t_p,x)|\\
	+ & 
	\frac{(1-\theta_2)(1+M)}{\theta_2}\frac{1}{J}\sum^J_{r=1-J}\left|z^\Delta_{p+1}\left(\frac{r}{2^m}\right)-\rho_{z^{\Delta}_{p+1}}\left(\frac{r}{2^m}\right)\right|
		(|\mathbb{E}^x_p[\varphi_{J,r}(X^\Delta_{t_{p+1}})]|+\mathbb{E}^x_p[\varphi_{J,r}(X^\Delta_{t_{p+1}})\Delta\omega_{p+1}]|)\\
	+ &
	\left(\frac{1}{\theta_2 \Delta t}+\frac{(1-\theta_2)M}{\theta_2}\right)\frac{1}{J}\sum^{J}_{r=1-J}\left|y^\Delta_{p+1}\left(\frac{r}{2^m}\right)-\rho_{y^\Delta_{p+1}}\left(\frac{r}{2^m}\right)\right||\mathbb{E}^x_p[\varphi_{J,r}(X^\Delta_{t_{p+1}})\Delta\omega_{p+1}]|
\end{align*}
Note that we have added some extra terms in the last inequality to simplify the expression. 
Taking the maximum over $\left\{\frac{(1-\theta_2)(1+M)}{\theta_2}, \frac{1}{\theta_2 \Delta t}+\frac{(1-\theta_2)M}{\theta_2}\right\}$ finishes the proof for Equation \eqref{recurring_z}.

Using Equations \eqref{picard_iteration_error}, \eqref{recurring_z}, \eqref{function_h}, \eqref{eq:approximation_function_h}, \eqref{eq:approximated_expectation},\eqref{error_formula}, and standing assumption (A3), we have
\begin{align*}
	& |\hat{y}^{\Delta,I}_p-y^\Delta_p(x)| \\
	\leq & 
	\frac{1+\xi}{1-\xi}\varepsilon^{Picard}_p+\frac{\xi}{1-\xi}|\hat{z}^\Delta_p(x)-z^\Delta_p(x)|+ \frac{1}{1-\xi}|\hat{h}(t_p,x)-h(t_p,x)|\\
	\leq &
	\frac{1+\xi}{1-\xi}\varepsilon^{Picard}_p
	+\frac{M_1\cdot \xi}{1-\xi}(|\zeta^{m,J}_{z^{\Delta}_{p+1}}(t_p,x)|+|\zeta^{m,J,\omega}_{y^{\Delta}_{p+1}}(t_p,x)|	+|\zeta^{m,J,\omega}_{f_{p+1}}(t_p,x)|)\\
	+ & \frac{M_1\cdot \xi}{1-\xi} \frac{1}{J}\sum^J_{r=1-J}\left|z^\Delta_{p+1}\left(\frac{r}{2^m}\right)-\rho_{z^{\Delta}_{p+1}}\left(\frac{r}{2^m}\right)\right|
		(|\mathbb{E}^x_p[\varphi_{J,r}(X^\Delta_{t_{p+1}})]|+\mathbb{E}^x_p[\varphi_{J,r}(X^\Delta_{t_{p+1}})\Delta\omega_{p+1}]|)\\
	+ & \frac{M_1\cdot \xi}{1-\xi}\frac{1}{J}\sum^{J}_{r=1-J}\left|y^\Delta_{p+1}\left(\frac{r}{2^m}\right)-\rho_{y^\Delta_{p+1}}\left(\frac{r}{2^m}\right)\right||\mathbb{E}^x_p[\varphi_{J,r}(X^\Delta_{t_{p+1}})\Delta\omega_{p+1}]|\\
	+ & \frac{1}{1-\xi}|\mathbb{E}^x_p[y^\Delta_{p+1}(X^\Delta_{t_{p+1}})]-\hat{\mathbb{E}}[y^\Delta_{p+1}(X^\Delta_{t_{p+1}})|X^\Delta_{t_p}=x]|\\
	+ & \frac{\Delta t (1-\theta_1)}{1-\xi}|\mathbb{E}^x_p[f(t_{p+1},X^\Delta_{t_{p+1}}, y^\Delta_{p+1}(X^\Delta_{t_{p+1}}), z^\Delta_{p+1}(X^\Delta_{t_{p+1}})]-\hat{\mathbb{E}}[f(t_{p+1},X^\Delta_{t_{p+1}},y^\Delta_{p+1}(X^\Delta_{t_{p+1}}), z^\Delta_{p+1}(X^\Delta_{t_{p+1}}))|X^\Delta_{t_p}=x]||\\
	\leq &
	\frac{1+\xi}{1-\xi}\varepsilon^{Picard}_p
	+\frac{M_1\cdot \xi}{1-\xi}(|\zeta^{m,J}_{z^{\Delta}_{p+1}}(t_p,x)|+|\zeta^{m,J,\omega}_{y^{\Delta}_{p+1}}(t_p,x)|	+|\zeta^{m,J,\omega}_{f_{p+1}}(t_p,x)|)\\
	+ & 
	\frac{1}{1-\xi}|\zeta^{m,J}_{y^\Delta_{p+1}}(t_p,x)| + \frac{\Delta t(1-\theta_1)}{1-\xi}|\zeta^{m,J}_{f_{p+1}}(t_p,x)|\\
	+ & 
	\frac{M_1\cdot \xi}{1-\xi} \frac{1}{J}\sum^J_{r=1-J}\left|z^\Delta_{p+1}\left(\frac{r}{2^m}\right)-\rho_{z^{\Delta}_{p+1}}\left(\frac{r}{2^m}\right)\right|
		(|\mathbb{E}^x_p[\varphi_{J,r}(X^\Delta_{t_{p+1}})]|+\mathbb{E}^x_p[\varphi_{J,r}(X^\Delta_{t_{p+1}})\Delta\omega_{p+1}]|)\\
	+ &
	\frac{M_1\cdot \xi}{1-\xi}\frac{1}{J}\sum^{J}_{r=1-J}\left|y^\Delta_{p+1}\left(\frac{r}{2^m}\right)-\rho_{y^\Delta_{p+1}}\left(\frac{r}{2^m}\right)\right||\mathbb{E}^x_p[\varphi_{J,r}(X^\Delta_{t_{p+1}})\Delta\omega_{p+1}]|\\
	+ &
	\frac{1}{1-\xi}\frac{1}{J}\sum^J_{r=1-J}\left|y^\Delta_{p+1}\left(\frac{r}{2^m}\right)-\rho_{y^{\Delta}_{p+1}}\left(\frac{r}{2^m}\right)\right||\mathbb{E}^x_p[\varphi_{J,r}(X^\Delta_{t_{p+1}})]|\\
	+ &
	\frac{\Delta t (1-\theta_1)}{1-\xi}\frac{1}{J}\sum^J_{r=1-J}\left|f\left(t_{p+1},\frac{r}{2^m}, y^\Delta_{p+1}\left(\frac{r}{2^m}\right), z^\Delta_{p+1}\left(\frac{r}{2^m}\right)\right)-\rho_{f_{p+1}}\left(\frac{r}{2^m}\right)\right||\mathbb{E}^x_p[\varphi_{J,r}(X^\Delta_{t_{p+1}})]|\\
	\leq &
	\frac{1+\xi}{1-\xi}\varepsilon^{Picard}_p
	+\frac{M_1\cdot \xi}{1-\xi}(|\zeta^{m,J}_{z^{\Delta}_{p+1}}(t_p,x)|+|\zeta^{m,J,\omega}_{y^{\Delta}_{p+1}}(t_p,x)|	+|\zeta^{m,J,\omega}_{f_{p+1}}(t_p,x)|)\\
	+ & 
	\frac{1}{1-\xi}|\zeta^{m,J}_{y^\Delta_{p+1}}(t_p,x)| + \frac{\Delta t(1-\theta_1)}{1-\xi}|\zeta^{m,J}_{f_{p+1}}(t_p,x)|\\
	+ &	
	\frac{M_1\cdot \xi + M \Delta t(1-\theta_1))}{1-\xi} \frac{1}{J}\sum^J_{r=1-J}\left|z^\Delta_{p+1}\left(\frac{r}{2^m}\right)-\rho_{z^{\Delta}_{p+1}}\left(\frac{r}{2^m}\right)\right|
		(|\mathbb{E}^x_p[\varphi_{J,r}(X^\Delta_{t_{p+1}})]|+\mathbb{E}^x_p[\varphi_{J,r}(X^\Delta_{t_{p+1}})\Delta\omega_{p+1}]|)\\
	+ & 
	\frac{M_1\cdot \xi+1+M \Delta t(1-\theta_1)}{1-\xi} \frac{1}{J}\sum^J_{r=1-J}\left|y^\Delta_{p+1}\left(\frac{r}{2^m}\right)-\rho_{y^{\Delta}_{p+1}}\left(\frac{r}{2^m}\right)\right|
		(|\mathbb{E}^x_p[\varphi_{J,r}(X^\Delta_{t_{p+1}})]|+\mathbb{E}^x_p[\varphi_{J,r}(X^\Delta_{t_{p+1}})\Delta\omega_{p+1}]|).
\end{align*}
This concludes the proof if $M_2 := \frac{M_1\cdot \xi+1+M \Delta t(1-\theta_1)}{1-\xi}$ and $M_3 := \frac{1+\xi}{1-\xi}\frac{\varepsilon^{Picard}_p}{M_2}$. Again, extra terms are included in the expression to simplify the formula.
\end{proof}

\subsection{A simple example for deriving the error formula}

In this section, we would use the result in Section \ref{errors of the recursive scheme} to derive an error formula for the approximation of Example 2 in Section \ref{section_example_2}.
In addition to the parameters provided in Section \ref{section_example_2}, we let $\theta_1 = 0$, $\theta_2 = 1$ and $P=2$.
Note that $P=2$ is merely used here to simplify our expression.

It is clear that driver function $f$ and terminal function $g$ satisfy all standing assumptions with the Lipschitz coefficient of $f$ with respect to $y$ and $z$ being 0.4.
We have $\Delta t = 0.05$, $\xi = \varepsilon^{Picard}_p = 0$ for $p = 0,1$ in our setting.
Using the derivation in Appendix B, Equations \eqref{recurring_z} and \eqref{recurring_y} can be simplified as follows:
\begin{align*}
	\frac{1}{20}|\hat{z}^\Delta_p(x) - z^\Delta_p(x)| 
	\leq & 
	|\zeta^{m,J,\omega}_{y^{\Delta}_{p+1}}(t_p,x)|
	+ \frac{1}{J}\sum^{J}_{r=1-J}\left|y^\Delta_{p+1}\left(\frac{r}{2^m}\right)-\rho_{y^\Delta_{p+1}}\left(\frac{r}{2^m}\right)\right||\mathbb{E}^x_p[\varphi_{J,r}(X^\Delta_{t_{p+1}})\Delta\omega_{p+1}]|\\
	\frac{1}{1.02}|\hat{y}^{\Delta,I}_p(x) - y^{\Delta}_p(x)|
	\leq & 
	|\zeta^{m,J}_{y^\Delta_{p+1}}(t_p,x)| + |\zeta^{m,J}_{f_{p+1}}(t_p,x)|\\
	+ &	
	\frac{1}{J}\sum^J_{r=1-J}\left|z^\Delta_{p+1}\left(\frac{r}{2^m}\right)-\rho_{z^{\Delta}_{p+1}}\left(\frac{r}{2^m}\right)\right|
		(|\mathbb{E}^x_p[\varphi_{J,r}(X^\Delta_{t_{p+1}})]|+\mathbb{E}^x_p[\varphi_{J,r}(X^\Delta_{t_{p+1}})\Delta\omega_{p+1}]|)\\
	+ & 
	\frac{1}{J}\sum^J_{r=1-J}\left|y^\Delta_{p+1}\left(\frac{r}{2^m}\right)-\rho_{y^{\Delta}_{p+1}}\left(\frac{r}{2^m}\right)\right|
		(|\mathbb{E}^x_p[\varphi_{J,r}(X^\Delta_{t_{p+1}})]|+\mathbb{E}^x_p[\varphi_{J,r}(X^\Delta_{t_{p+1}})\Delta\omega_{p+1}]|),
\end{align*}
for $p = 0, 1$.
Therefore, the error of applying the quick SWIFT scheme to the discretizated system reads:
\begin{align*}
	&
	|\hat{y}^{\Delta, I}_0(x_0)-y^\Delta_1(x_0)|\\
	\leq & 
	1.02|\zeta^{m,J}_{y^\Delta_{1}}(0,x_0)| + 1.02 |\zeta^{m,J}_{f_{1}}(0,x_0)|\\
	+ &	
	\frac{1.02}{J}\sum^J_{r=1-J}\left|z^\Delta_{1}\left(\frac{r}{2^m}\right)-\hat{z}^{\Delta}_{1}\left(\frac{r}{2^m}\right)\right|
		(|\mathbb{E}^{x_0}_0[\varphi_{J,r}(X^\Delta_{t_{1}})]|+\mathbb{E}^{x_0}_0[\varphi_{J,r}(X^\Delta_{t_{1}})\Delta\omega_{1}]|)\\
	+ & 
	\frac{1.02}{J}\sum^J_{r=1-J}\left|y^\Delta_{1}\left(\frac{r}{2^m}\right)-\hat{y}^{\Delta,I}_{1}\left(\frac{r}{2^m}\right)\right|
		(|\mathbb{E}^{x_0}_0[\varphi_{J,r}(X^\Delta_{t_{1}})]|+\mathbb{E}^{x_0}_0[\varphi_{J,r}(X^\Delta_{t_{1}})\Delta\omega_{1}]|)\\
			\leq & 
	1.02|\zeta^{m,J}_{y^\Delta_{1}}(0,x_0)| + 1.02|\zeta^{m,J}_{f_{1}}(0,x_0)|\\
	+ & 
	\frac{1}{J}\sum^J_{r=1-J}\left(
		20.4\left|\zeta^{m,J,\omega}_{y^{\Delta}_{2}}\left(t_1,\frac{r}{2^m}\right)\right| 
		+ 1.02^2\left|\zeta^{m,J}_{y^\Delta_{2}}\left(t_1,\frac{r}{2^m}\right)\right| 
		+ 1.02^2\left|\zeta^{m,J}_{f_{2}}\left(t_1,\frac{r}{2^m}\right)\right|
	\right)\\
	& \hspace{15mm}
	(|\mathbb{E}^{x_0}_0[\varphi_{J,r}(X^\Delta_{t_{1}})]|+\mathbb{E}^{x_0}_0[\varphi_{J,r}(X^\Delta_{t_{1}})\Delta\omega_{1}]|).
\end{align*}
Note that there is no recurring error at $t_2$ as we know the exact terminal condition.

\newpage

\begin{figure}
\centering
\begin{subfigure}{\textwidth}
    \centering
        \includegraphics[width=\textwidth]{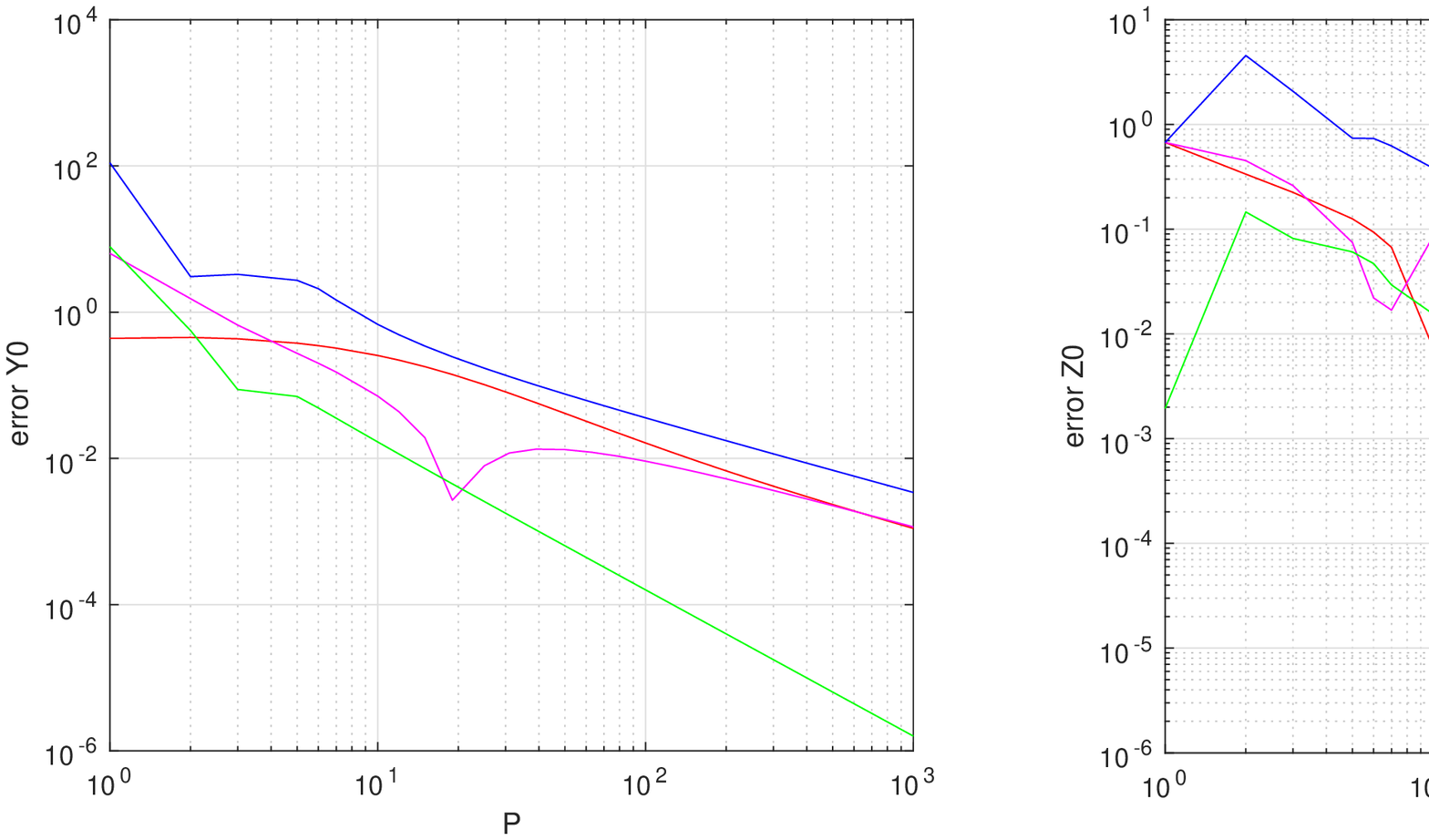}
    \caption{Quick SWIFT with $J=2^9$}
    \label{quick_SWIFT_example1}
\end{subfigure}%

\begin{subfigure}{\textwidth}
    \centering
        \includegraphics[width=\textwidth]{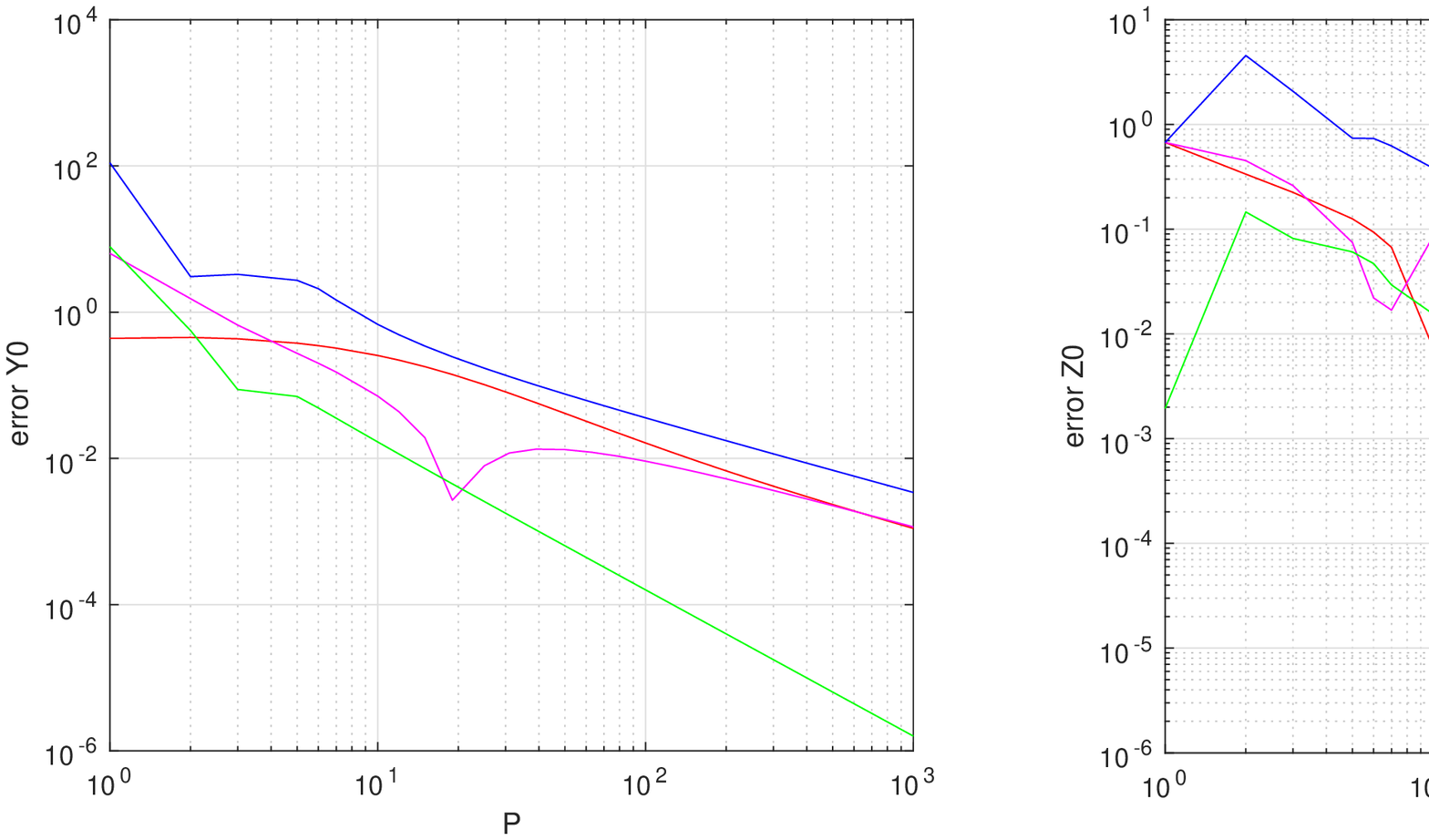}
    \caption{Mixed SWIFT with $J=2^9$}
    \label{mixed_quick_SWIFT_example1}
\end{subfigure}
\caption{Results example 1, left: error in $y(0,x_0)$, right: error in $z(0,x_0)$}
\end{figure}

\begin{figure}
\centering
\begin{subfigure}{\textwidth}
    \centering
        \includegraphics[width=\textwidth]{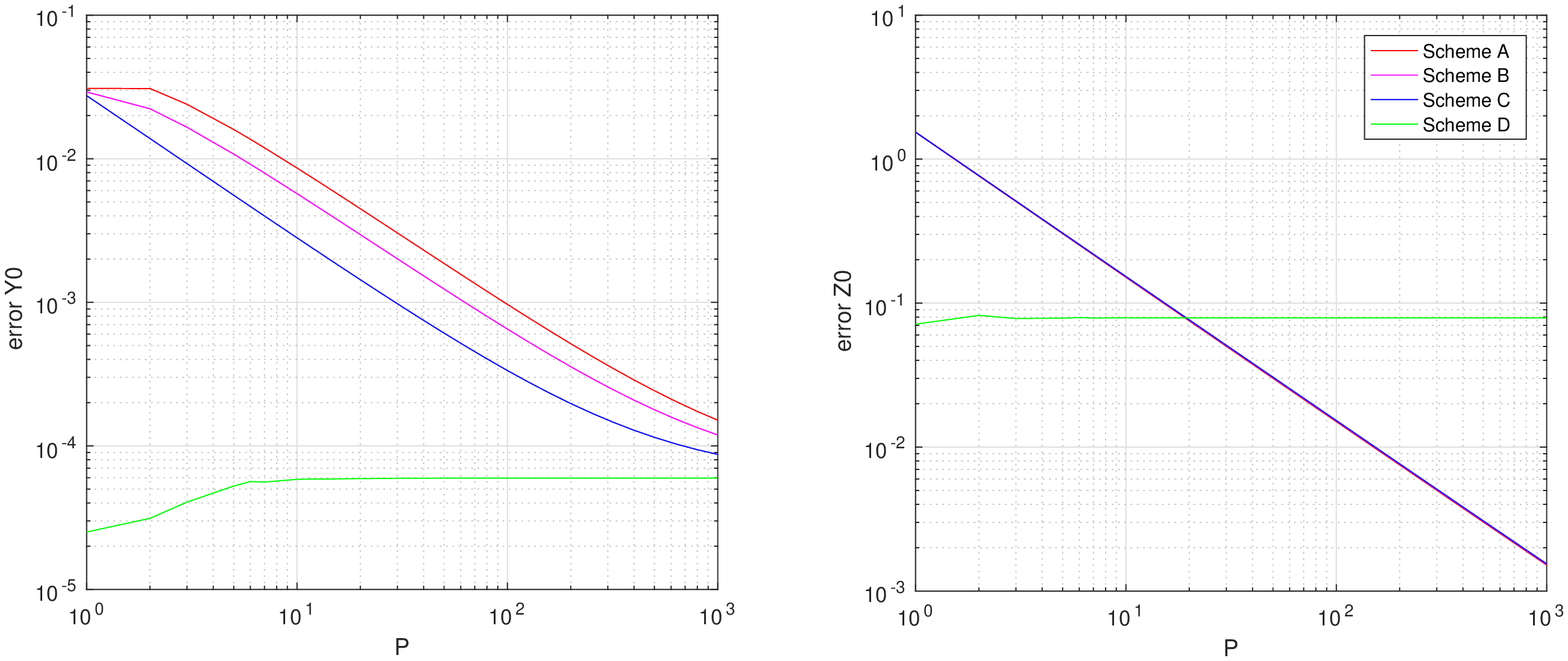}
    \caption{Quick SWIFT with $J=2^{9}$}
    \label{quick_SWIFT_example2}
\end{subfigure}%

\begin{subfigure}{\textwidth}
    \centering
     \includegraphics[width=\textwidth]{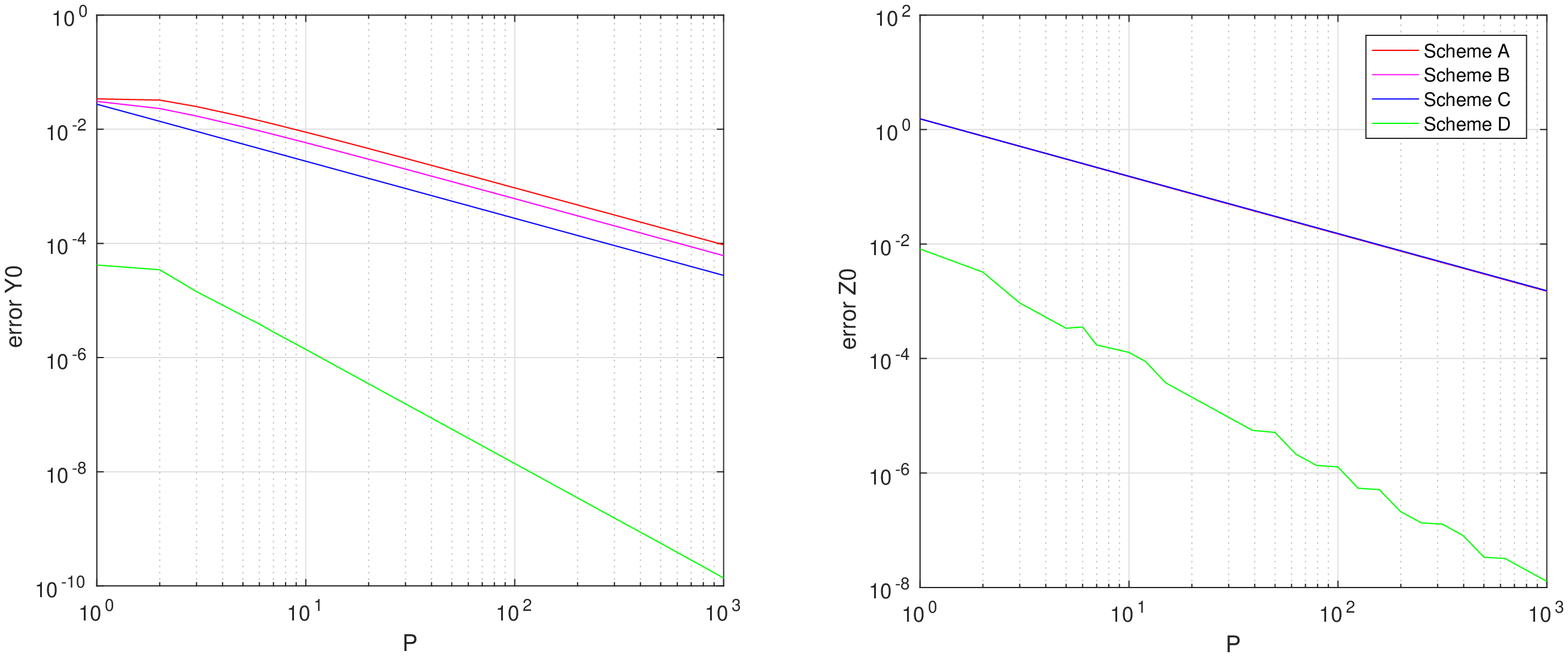}
    \caption{Mixed  SWIFT with $J=2^9$}
    \label{mixed_quick_SWIFT_example2}
\end{subfigure}
\caption{Results example 2, left: error in $y(0,x_0)$, right: error in $z(0,x_0)$}
\end{figure}

\begin{figure}
\centering
\includegraphics[width=\textwidth]{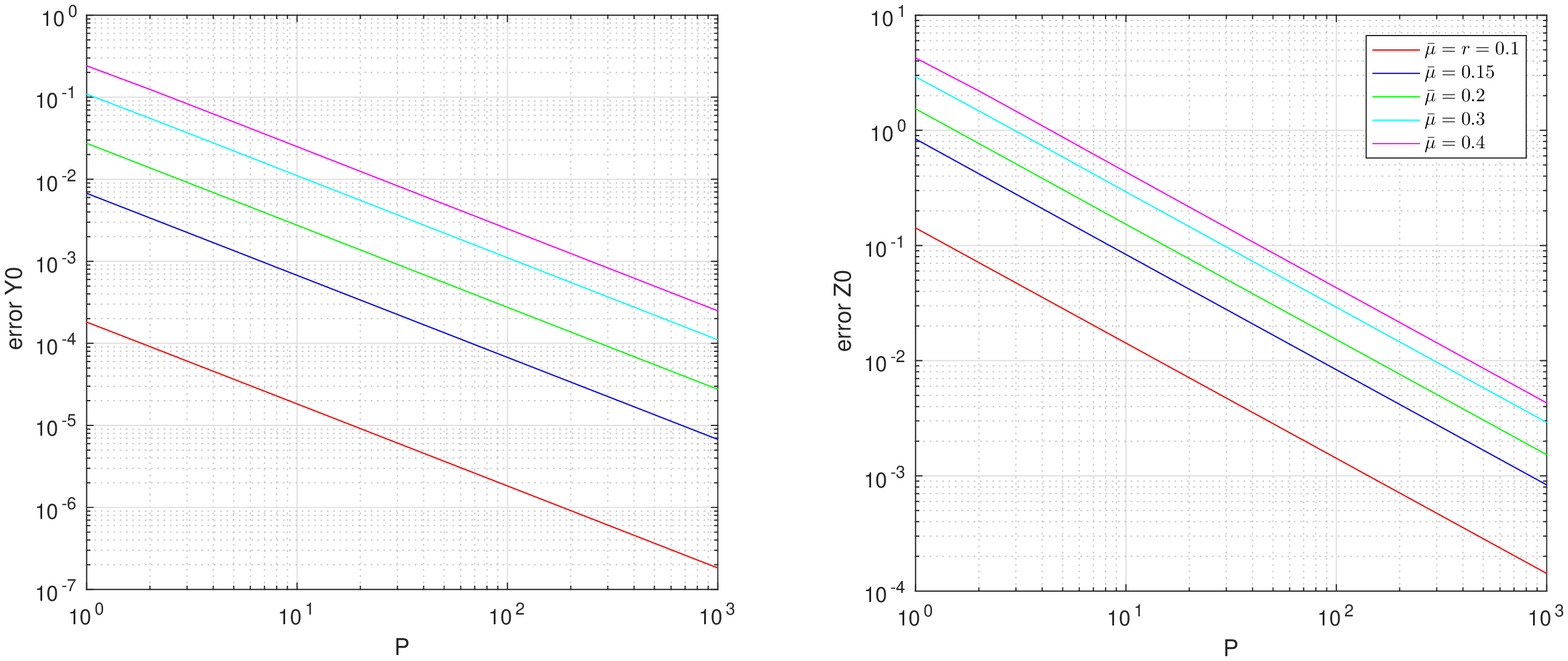}
\caption{Results example 2 for different values of $\bar{\mu}$ (Scheme C)}
\label{qexample2_mu}
\end{figure}

\begin{figure}
\centering
\includegraphics[width=\textwidth]{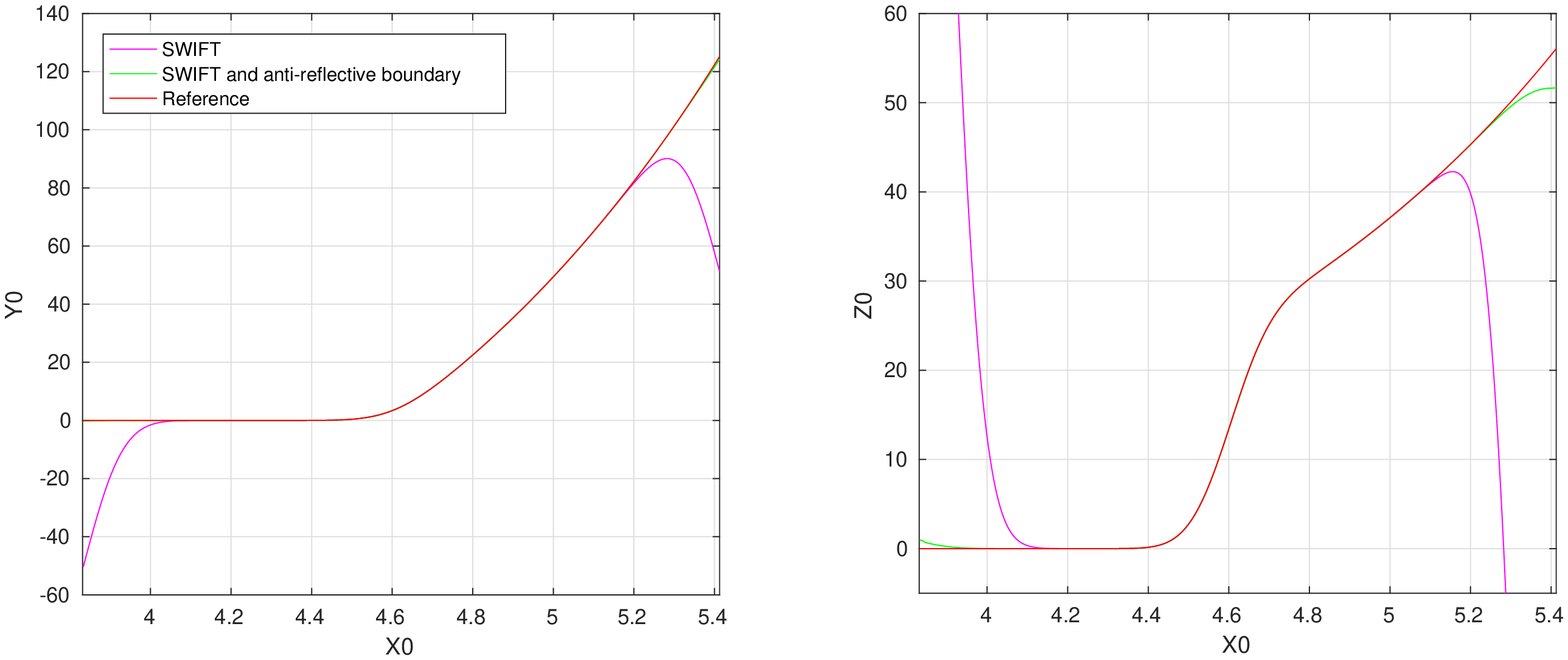}
\caption{Results example 2 with and without applying anti-reflective boundary technique (Scheme D, $P=1000$ and $J=2^9$)}
\label{qexample2_anti-relective}
\end{figure}

\begin{figure}
\centering
\begin{subfigure}{\textwidth}
    \centering
    \includegraphics[width=\textwidth]{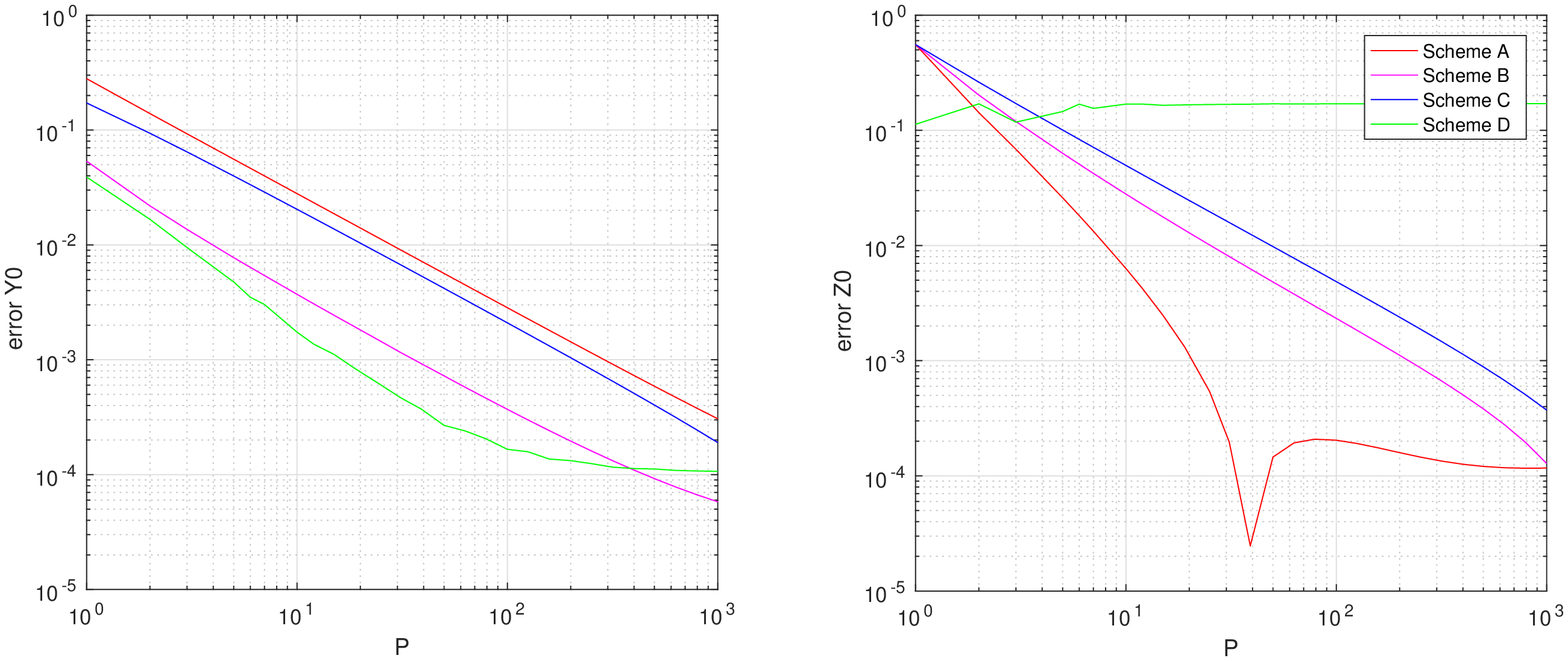}
    \caption{Quick SWIFT with $J=2^{9}$}
    \label{quick_SWIFT_example3a}
\end{subfigure}%

\begin{subfigure}{\textwidth}
    \centering
    \includegraphics[width=\textwidth]{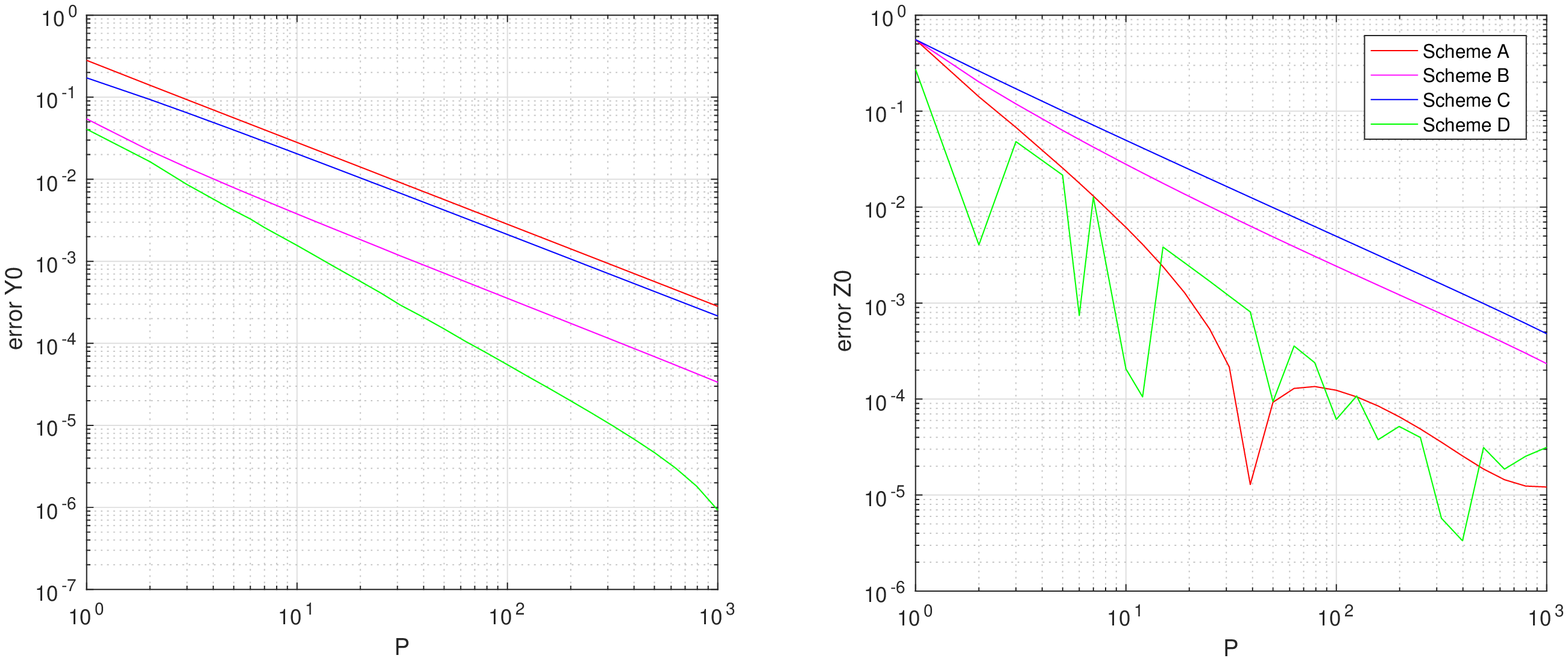}
    \caption{Mixed SWIFT with $J=2^9$}
    \label{mixed_quick_SWIFT_example3a}
\end{subfigure}
\caption{Results example 3, left: error in $y(0,x_0)$, right: error in $z(0,x_0)$}
\end{figure}

\begin{figure}
\centering
\includegraphics[width=\textwidth]{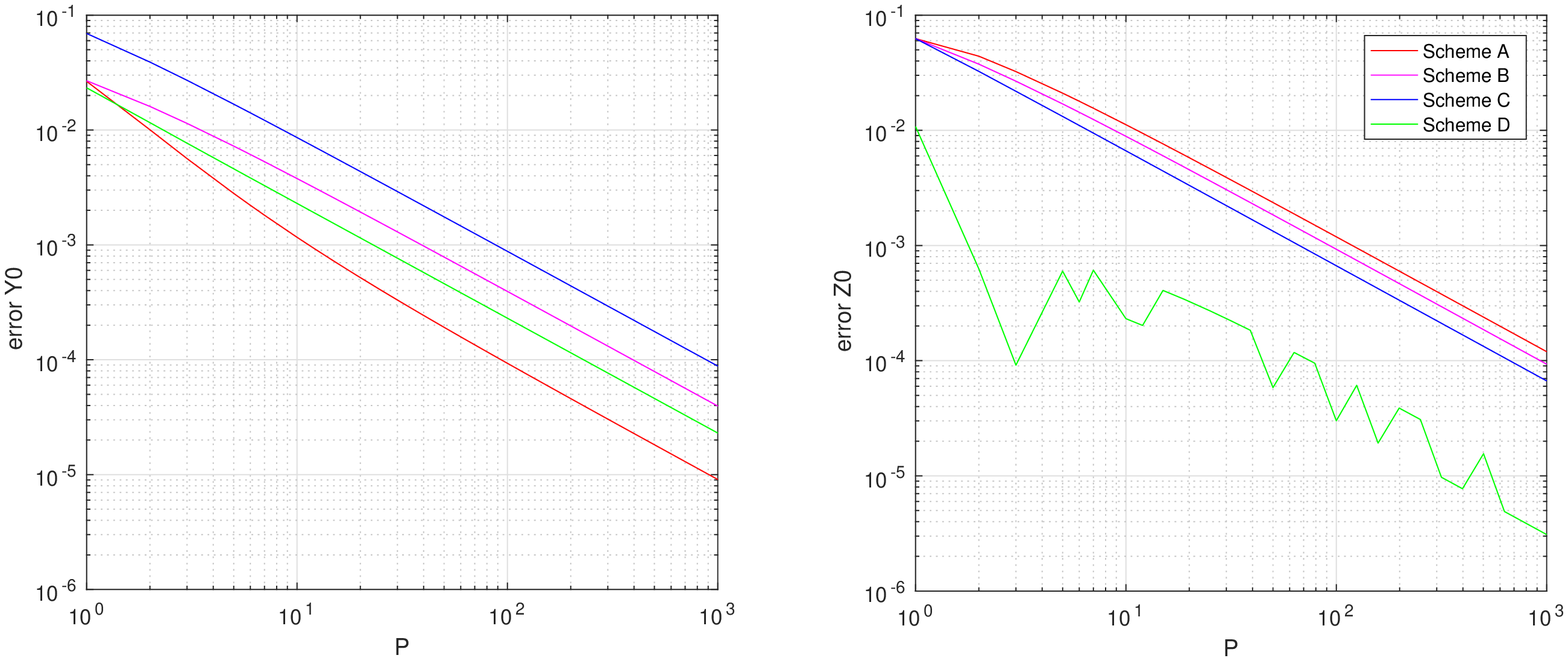}
\caption{Results example 4, ($J=2^9$), left: error in $y(0,x_0)$, right: error in $z(0,x_0)$}
\label{quick_SWIFT_example4}
\end{figure}

\end{document}